\documentclass[%
twocolumn,
superscriptaddress,
showpacs,preprintnumbers,
amsmath,amssymb,
aps,
pra,
]{revtex4-1}

\usepackage{graphicx}% Include figure files
\usepackage{dcolumn}% Align table columns on decimal point
\usepackage{bm}% bold math
\usepackage{colortbl}
\usepackage{epstopdf}
\usepackage{color}
\def\BibTeX{{\rm B\kern-.05em{\sc i\kern-.025em b}\kern-.08em
    T\kern-.1667em\lower.7ex\hbox{E}\kern-.125emX}}

\usepackage[utf8]{inputenc} % allow utf-8 input
\usepackage[T1]{fontenc}    % use 8-bit T1 fonts
\usepackage{hyperref}       % hyperlinks
\usepackage{url}            % simple URL typesetting
\usepackage{booktabs}       % professional-quality tables
\usepackage{nicefrac}       % compact symbols for 1/2, etc.
\usepackage{microtype}      % microtypography
\usepackage{xcolor}         % colors
\usepackage[normalem]{ulem}         % strikethrough text

\usepackage{float}					% For using float environments
\usepackage{subcaption}
\usepackage{amsthm}
\DeclareMathOperator*{\argmax}{arg\,max}

\DeclareMathOperator*{\maxi}{max}

\usepackage[ruled]{algorithm2e}
\usepackage{multirow}
\definecolor{red}{rgb}{0.882, 0.281, 0.281}
\definecolor{green}{rgb}{0.0, 0.0, 0.0} 
\definecolor{orange}{rgb}{0.875, 0.472, 0.093}
\newcommand{\red}[1]{\textcolor{red}{#1}}
\newcommand{\green}[1]{\textcolor{green}{#1}}
\newcommand{\orange}[1]{\textcolor{orange}{#1}}

\newcommand{\E}[1]{\mathbb E\left[#1\right]}
\newcommand{\PP}[1]{\mathbb P\left[#1\right]}

\theoremstyle{definition}

\theoremstyle{plain}

\newtheorem{theorem}{Theorem}[section]
\newtheorem{corollary}{Corollary}[theorem]
\newtheorem{lemma}[theorem]{Lemma}

\newcommand{\floor}[1]{\lfloor #1 \rfloor}

\begin{document}

\title{A study of first-passage time minimization via Q-learning in heated gridworlds}
\author{M.~A.~Larchenko}
\affiliation{Center for Computational and Data-intensive Science and Engineering, Skolkovo Institute of Science and Technology, 121205, Moscow, Russia}
\author{P.~Osinenko}
\affiliation{Center for Computational and Data-intensive Science and Engineering, Skolkovo Institute of Science and Technology, 121205, Moscow, Russia}
\author{G.~Yaremenko}
\affiliation{Center for Computational and Data-intensive Science and Engineering, Skolkovo Institute of Science and Technology, 121205, Moscow, Russia}
\author{V.~V.~Palyulin}
\affiliation{Center for Computational and Data-intensive Science and Engineering, Skolkovo Institute of Science and Technology, 121205, Moscow, Russia}

\begin{abstract}
Optimization of first-passage times is required in applications ranging from nanobots navigation to market trading. In such settings, one often encounters unevenly distributed noise levels across the environment. We extensively study how a learning agent fares in 1- and 2- dimensional heated gridworlds with an uneven temperature distribution. The results show certain bias effects in agents trained via simple tabular Q-learning, SARSA, Expected SARSA and Double Q-learning. While high learning rate prevents exploration of regions with higher temperature, low enough rate increases the presence of agents in such regions. The discovered peculiarities and biases of temporal-difference-based reinforcement learning methods should be taken into account in real-world physical applications and agent design.
\end{abstract}

\maketitle

\section{Introduction}
\label{sec:introduction}
Machine learning methods have become established and widely used for solving many hard problems such as image or speech recognition \cite{vision1,vision2,speech}, games like Atari \cite{atari} and Go \cite{Go} or vehicle routing problems \cite{NIPS2020VRP,VRPReview}.
The latter few applications demonstrate, in particular, the emerging success of reinforcement learning (RL) approaches. Along with other machine learning approaches \cite{LenkaReview2019} the RL techniques find ever more applications in physics, especially for optimization of motion in complex physics environments \cite{carleo2019review,cichos2020review,hydrodynamics}.
However, tuning of RL agents is a non-trivial task and unexpected effects, such as biases \cite{hasselt2010double,sutton2018}, may occur in their deployment.

One area where RL is a major candidate for the development of autonomous navigation is active matter research \cite{cichos2020review}.
Active particles or agents are objects with an ability to control some of their dynamics and, thus, are a natural sandbox for RL algorithms.
A lion's share of the relevant work in active matter deals with small scales, where thermal fluctuations along with Brownian motion and turbulence play a crucial role \cite{modelsactivematter}.
They have to be taken into account while learning and optimizing control strategies.
RL was also applied for the navigation of microswimmers in such highly stochastic environments as complex and turbulent flows in \cite{biferale2019zermelo, colabrese2017flow,biferale2018PRLiquids,collectiveswimming2018,muinos2021reinforcement}.
The actor-critic RL significantly outperformed a trivial policy of finding the fastest path from A to B for an agent with a constant slip speed in 2D turbulent environment \cite{biferale2019zermelo}.
RL agents in 3D stationary Arnold-Beltrami-Childress helical flow learned to target specific high vorticity regions \cite{biferale2018PRLiquids}.
Among tabular RL methods, Q-learning is perhaps one of the most convenient.
It was used, e.~g., to control self-thermophoretic active particles in a solution with the real-time microscopy system\cite{muinos2021reinforcement}. The Q table corresponded to the discretized position of the microswimmer, thus, staging the gridworld geometry in experiment.
Munios et al. \cite{muinos2021reinforcement} noted that the noise due to Brownian motion substantially affects both the learning process and the actions within the learned behavior. 

Navigation and prediction of motion in highly stochastic or turbulent environments is a necessity not only for nanobots \cite{cichos2020review,KUKREJA2021,muinos2021reinforcement}, but even for large macroworld objects such as marine vehicles in ocean currents \cite{ocean2016}. The macroscopic movement optimization in turbulent media with RL was performed with gliders in turbulent air flows \cite{reddy2016learning,reddy2018glider}. The results clearly show that the efficiency of control decreases with an increasing speed of the glider which is equivalent to increased fluctuations. Still the learned soaring strategy was effective even in the case of strong fluctuations. Other relevant studies include Q-learning for the optimization of collective motion in stochastic environments with small UAVs learning how to flock (``Q-flocking") \cite{hung2016q}, deep RL for coordinated collective energy-saving swimming \cite{collectiveswimming2018}  and navigation with obstacle avoidance in a system with thermal fluctuations by using deep double reinforcement learning \cite{yang2020efficient}. The huge impact of stochastic dynamics, however, is not an exclusive speciality of physical systems. It underlies the base of the modern economy through stock price fluctuations on financial markets, where RL is expanding its presence as a trading algorithm \cite{financial,PENDHARKAR20181}.

\textbf{Randomness in RL}.
Generally, theory of RL and Markov decision processes as well as other control strategies employ noise as a part of the problem setting. 
Talking about applications, in the literature one could identify several directions concerning the impact of stochasticity:

\begin{itemize}
    \item \textit{Noisy reward signal.} Many problems such as games of chance have this noise as a feature \cite{hasselt2010double,hu2018accurate,he2019interleaved,}. Alternatively it comes from an imperfect observation process \cite{perturbed2020,everitt2017corruptedreward}.
    In human-guided learning, for instance, it arises from mistakes and incoherent answers of human teachers \cite{loftin2014humanfeedback}. 
    \item \textit{Measurement noise.} Any distortion or thermal fluctuations in actuators and sensors can affect both perceived data and agent's actions \cite{schoettler2020insertiontask, gullapalli1994peginhole, johannink2019residual, howell1997vehiclesuspension}.
    \item \textit{Noisy adversarial attacks.} Sometimes security problems appear during training process and attacks against pre-trained agents \cite{Huang2017AdversarialAO,Huang2019DeceptiveRL,Gleave2020AdversarialPA}.
    \item \textit{Learning in stochastic dynamics.} The transition between states is affected by a random force and the environment dynamics is assumed to be stochastic \cite{baird1994reinforcement, fox2015taming}.
\end{itemize}

Depending on the problem, different adjustments to a learning process or algorithms were suggested. 
For instance, Double Q-learning \cite{hasselt2010double} is a modification of Q-learning \cite{Qlearning} with a double estimator that counters maximization bias and demonstrates superiority in tasks with noisy reward signal. 
However, there are no examples with stochastic transition dynamics in the original paper. 
Impact of stochastic transition function was discussed in \cite{fox2015taming}, where G-learning was proposed and tested against both Q- and Double Q-learning.
Advantage updating presented in \cite{baird1994reinforcement} was compared with Q-learning in linear quadratic regulator problem with presence of noise in state transition function.

However, to the best of our knowledge a comparison of different algorithms with respect to learning in presence of stochastic dynamics of states was not carried out yet carefully. Thus, it is unclear what kind of adjustments could be useful in such settings.

\textbf{First-passage problem}.
Over the past couple of decades it became clear that many of search and optimization problems in physics, biology and finance can be formulated within a first-passage framework \cite{redner01,Grebenkov_2020}. The first-passage time (FPT) $\tau$ is a first moment when an agent/process with a coordinate/value $X_t$ reaches a boundary $x$, i.~e. $\tau= \mathrm{inf}\{t > 0 :X_t\ge x\}$.
This boundary could be, for instance, a threshold to sell or buy at a stock market or a location of reward in space or a site which, once reached, triggers a biological/chemical process in a living cell. Alternatively, one could reformulate this approach in terms of survival times and probabilities. The FPT optimization is often done by analytical or numerical minimization within model's assumptions as in many of the ecological \cite{viswanathan}, biological \cite{Mirny_2009} problems or risk assessment \cite{risk}. First-passage times
have been found to be connected to the relative advantage of states in Markov decision processes (MDP) \cite{stoshasticgames} and have proven to be useful for characterization of reachability of states \cite{reachability}. The passage time itself could be used as a reward function for an algorithm to minimize. It is interesting that in more traditional thermal bath settings the minimization of FPT could produce non-trivial results such as complex shapes of potentials needed for minimization observed in theory and experiment \cite{Palyulin_2012,Chupeau1383}. 

In this paper we find that uneven noise distributions can trigger biases of RL learning algorithms and as such have to be paid attention to. We introduce a new type of gridworld models with a state-dependent noise affecting actions which we call as heated gridworlds. We perform an extensive study of agent learning in heated gridworlds with state-dependent temperature distribution. We find that the state-dependency of the noise triggers convergence of agents to suboptimal solutions, around which the respective policies stay for practically long learning times. This happens with such common RL algorithms as tabular Q-learning, SARSA, Expected SARSA and Double Q-learning. The observed phenomena should be taken into account in design and deployment of agents in physical applications that follow the formalism of a heated gridworld.

\textbf{Notation}. Capital letters will denote random variables, if not specified otherwise.
Small letters will denote definite values thereof.
For instance, if $R_t$ is a reward at time $t$ as a random variable, then $r_t$ is a value that it assumed.

\section{Background}
\label{sec:background}

The general scheme of RL consists of an agent and an environment. The agent interacts with an environment in a cycle by doing actions and receiving rewards \cite{sutton2018,bertsekas1996neuro}.
RL problem is usually described with the framework of Markov Decision Processes (MDP).
At a time frame $t$, the agent perceives an environment through a random state $S_t \in \mathcal S$, then selects an action $A_t \in \mathcal A$ distributed with a probability distribution $\pi$, called the \textit{policy}, and gets some $\mathcal R \subseteq \mathbb R$-valued random reward $R_t$.
The sets $\mathcal S, \mathcal A, \mathcal R$ are assumed to be finite. The environment transitions into the next state adopt a value $s_{t+1}$ with probability $P_S(s_{t+1}| s_t, a_t) = \PP{S_{t+1}=s_{t+1}| S_t=s_t, A_t=a_t}$, where $s_t, a_t$ are the current state and action values, respectively. $a_t$ is sampled from the probability distribution $\pi$. The policy is assumed to be Markovian $\pi(a_t) = \pi(a_t|s_t)$. The total expected discounted reward under a policy $\pi$ starting at a state $s$ is denoted as

\begin{equation}
    \label{eqn:reward-to-go}
    V^\pi(s) = \mathbb{E}_\pi\Big[ \sum_{t} \gamma^t r_{t} \Big|S = s \Big],
\end{equation}
where $\gamma \in (0, 1)$ is a discounting factor.

Another handy formalism is of action-value functions. If an agent takes an action $a$ in a state $s$ and follows $\pi$ thereafter one can define
\begin{equation}
    \label{eqn:Q-pi}
    Q^\pi(s, a) = \mathbb{E}_\pi\Big[ \sum_{t} \gamma^t r_{t} \Big|S = s, A = a\Big].
\end{equation}  

The agent's goal is to find the optimal policy $\pi^\star$ that maximizes the discounted sum of received rewards
\begin{align*}
    Q^\star(s, a) = \max_\pi Q^\pi(s, a), \\
    \pi^\star(s) = \argmax_a Q^\star(s, a).
\end{align*}

\paragraph{First-passage time minimization and MDP}

If the task is in finding the fastest way to a target, one can tie the reward signal to the time needed to reach the target.
The value function could then be made proportional to the mean first-passage time (MFPT). Hence, the optimization of the policy is equivalent to minimization of the MFPT.
% Technically speaking, let us fix $(\Omega, \Sigma, \mathbb P)$.
Technically speaking, one can write the following dynamical system,
\begin{equation}
\label{eqn:gridworld-env}
  \begin{array}{l}
        \tilde{S}_{t + 1} = \tilde{S}_t + \pi(\tilde{S}_t) + W_t, \\
        S_{t + 1} = \tilde{S}_{t + 1}, S_t, \tilde{S}_{t + 1} \in \mathcal S^\dagger \subseteq \mathcal S \subset \mathbb Z^n, \\
        S_{t + 1} = s_*, \text{ otherwise}, \\
        W_t \sim \mathcal W(\tilde{S}_t), \\
        % \forall k \in \mathbb{Z} \ : \ \pi(k) \in \left\{-1, 0, 1 \right\} 
        S_0 = \tilde{s}_0 \in \mathbb{Z}^n,
    \end{array}
\end{equation}
where $\mathcal S^\dagger$ represents the gridworld, and $s_*$ is an artificial formal state value which indicates the ``end of the game'' (the episode is considered finished if the environment state enters this value), $W_t$ represents the temperature effect as a random disturbance with a discrete probability distribution $\mathcal W(\theta)$ having a finite support as a bounded subset of $\mathbb{Z}^n$ and with parameters $\theta$ which are random. The random state $\tilde{S}_t$ and the policy $\pi$ take values from a bounded subset of $\mathbb{Z}^n$. $\pi$ is a Markov policy.
That is, at each time frame $t$, the agent takes some finite number of steps in each direction on $n$-dimensional grid.
Then its actions are perturbed by temperature effects $W_t$ that shift the agent randomly by a finite number of steps.
In general, the respective probability distribution depends on the current state.
When the agent crosses the boundary and thus leaves the gridworld $\tilde{S}_t$, we formally ``fix'' the state at an abstract value $s_*$.
% $\partial \mathcal S^\dagger$ is its boundary
%where $\mathcal B(a, b)$ denotes the binomial distribution with parameters $a, b$.
%Let $\Omega$ be the sample space of this stochastic model.
% Let $\bar{\mathcal S}$ denote $[-\zeta, \zeta]\cap\mathbb{Z}$ and let $\mathcal S$ denote $\bar{\mathcal S}\cup\{x_*\}$.

An equivalent description of the above gridworld reads:
\begin{equation}
    \label{eqn:gridworld-env-alt}
    \begin{array}{ll}
        S_{t + 1} = S_t + \pi(S_t) + W_t, & S_t \in \mathcal S^\dagger \land \\
        & S_t + \pi(S_t) + W_t \rvert \in \mathcal S^\dagger\\
        S_{t + 1} = s_*, & \text{otherwise}.
    \end{array}
\end{equation}

Formally, $S$ and $\tilde{S}$ are states of two independent dynamical systems that can each be considered on their own.
The dynamical system corresponding to $S$ can be understood as ``absorbing'', whereas the one with $\tilde{S}$ as ``free''.
Essentially, the trajectories of the two systems coincide up until the first passage beyond $\mathcal S^\dagger$.

We can formulate the reward function for the considered first passage problem as follows:
\begin{equation}
    \label{eqn:reward-gridworld}
    r(s_t, \pi(s_t), w_t) = \mathbb I_{\{s_*\}}(s_{t + 1})\mathbb I_{\mathcal S^\dagger}(s_t)r_{\text{target}} - \mathbb I_{\mathcal S^\dagger}(s_t),
\end{equation}
where $\mathbb I$ is the indicator function, $r_{\text{target}} \in \mathbb R$ is any number used purely for scoring, i.e. when the agent is still inside $\mathcal S^\dagger$, it gets a minus one point. When it crosses the boundary, it receives some $r_{\text{target}}$ points, which may conveniently be set to a large number compared to the grid size. The goal is to score as many points as possible, i.e. to cross the boundary as fast as possible.
Let $R_{\text{tot}}$ be the random variable of total reward and let $J$ be the objective function as the expected total reward,
\begin{equation}
    \label{eqn:objective-fnc}
    J[\pi] = \E{\sum_{t = 0}^\infty r(S_t, \pi(S_t), W_t)} = \E{R_{\text{tot}}}.
\end{equation}
The next subsection describes particular algorithms tested for training boundary-crossing agents in this work, whereas a more detailed theoretical analysis of the gridworld is provided in the Appendix \ref{theory} where it is shown that the respective objective is well-defined for any admissible policy.

\paragraph{Temporal-difference algorithms}

When the transition probability function $P_S$ is known and can be analytically expressed, a solution to RL problem can be obtained from Bellman equation \cite{bellman1957dynamic} or Hamilton–Jacobi–Bellman equation.
Often it is not the case and agents have to learn through interaction, building an estimation of the value function on-line. 
Temporal-difference TD(0) algorithms (here we use Q-learning, SARSA, Expected SARSA and Double Q-learning) employ this idea, renewing the estimation $Q(s, a)$ after every time frame. 
For all listed algorithms $Q(s, a)$ is stored in Q table. Update rules for Q table for all used methods are given in Algorithms \ref{algo:q-learning}, \ref{algo:sarsa}, \ref{algo:expected_sarsa} and \ref{algo:double_learning} (see Appendix A).

\paragraph{Choice of learning rate}
The update rule of Q-learning is governed by a learning rate $\alpha$ and a discount rate $\gamma$ as follows:
\begin{align*}
& Q(s_t, a_t) := \\ 
& Q(s_t, a_t) + \alpha\big(r_{t} +\gamma \maxi_a Q(s_{t+1},a) - Q(s_t, a_t)\big).
\end{align*}
where $r_{t} +\gamma \maxi_a Q(s_{t+1},a)$ is the update target, in other words, let us define:
\begin{align*}
 \text{Update}_t := r_{t} +\gamma \maxi_a Q(s_{t+1}, a).
\end{align*}
There is one more tuning parameter called exploration-exploitation parameter $\varepsilon$, which is not used in update rule directly but affects the exploratory behaviour and is required for a proper convergence of the algorithm (see the algorithm and $\varepsilon$-greedy policy description in Appendix \ref{algorithms}).

As it was shown in \cite{watkins1992q}, Q-learning converges to $Q^\star(s, a)$ and the optimal policy if each state-action pair is visited infinitely many times ($\varepsilon$-greedy policy) and the learning rate satisfies the conditions
\begin{align*}
    \sum_t \alpha_t = \infty, \ \ \ \ \ \ \ \  \sum_t \alpha_t^2 < \infty.
\end{align*}

However, in practice both constant \cite{NIPS2020VRP, reddy2016learning, hasselt2010double, baird1994reinforcement} and scheduled \cite{yang2020efficient} learning rates are used as well.
The value $Q(s, a)$ for some fixed $(s, a)$ pair is renewed in a cycle \cite{sutton2018}
\begin{align*}
    Q_{n+1} &= Q_n + \alpha\big(\text{Update}_n - Q_n\big) \\
            &=(1-\alpha)^n Q_{1} + \alpha\sum_{i=1}^{n}(1-\alpha)^{n-i}\text{Update}_{i}.
\end{align*}
With a constant rate the weight of a single update in total sum decreases exponentially with the number of updates $n$. The higher $\alpha$ is, the sooner an agent overwrites its previous experience. We expect that in stochastic dynamics there could be a restriction on $\alpha$ under which convergence to certain policies is possible.

\section{Simulation framework}

% lets compare setup from this paper in the section with HeatedGridworld, not here.

\begin{figure}[H]
\centering
 \begin{subfigure}{\linewidth}
 \centering
  \includegraphics[width=\linewidth]{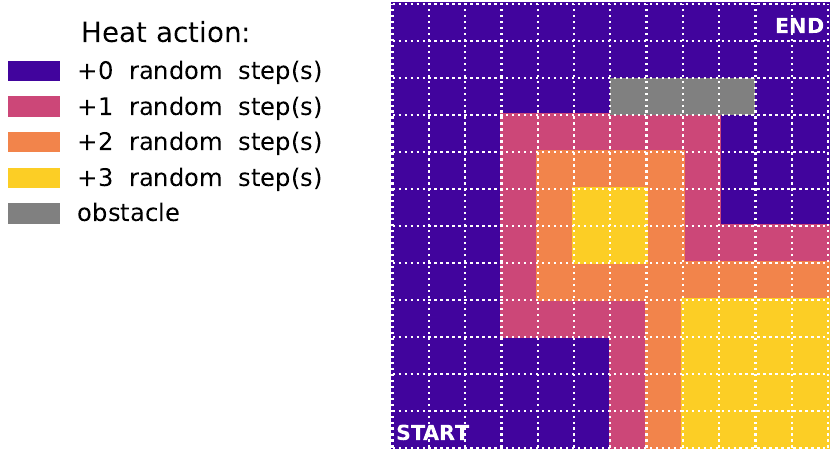} 
  \caption{}
  \label{fig:heated_gridworld_example}
 \end{subfigure}
 \begin{subfigure}{0.5\linewidth}
 \centering
  \includegraphics[width=\linewidth]{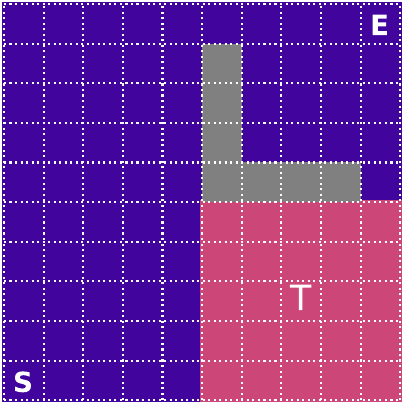}
  \caption{}
  \label{fig:heated_gridworld}
 \end{subfigure}
 \caption{(a) A heated gridworld example (see the description in the main text).
 (b)  10x10 heated gridworld considered in simulations of Section \ref{sec:2d-sim}.
 The heated region (with temperature) is colored in magenta.
 The considered temperature values are $T=0,1,2,3.$}
\label{fig:2D_case}
\end{figure}

% We consider an environment which we called as a "heated gridworld".
% Our main requirements to this setup are:
% - minimal modification of a gridworld
% - mimic thermal noise
% - ability to regulate noise intensity
% - keep things simple and interpretable
% 
% Using reasoning above, it was decided to add fixed number T of random actions after an action performed by an agent.
% This approach results in following scattering of the final position
%
% scattering_scheme.pdf

We introduce a new type of gridworlds which we call as heated gridworlds in order to test and to compare work of algorithms in the case of uneven, possibly state-dependent, distributed noise. The noise could be caused, for instance, by thermal fluctuations. As an example we sketch a 2-dimensional heated gridworld in Fig. \ref{fig:2D_case}. It is based on the common gridworld setup with 4 actions (left, right, up, down).
Every played time frame $t$ is penalized with $r_{tick}$ and when the goal position is reached the reward is $r_{target}$. 
% \textcolor{blue}{Note that minimised quantity depends only on $r_{tick}$. $r_{target}$ could be set to zero, which in some cases worsens the score but does not impact the result in principle}.
Thus, $r_{target}$ is used rather for convenience as discussed in Section \ref{sec:background}.
Attempts to cross the boundaries or obstacles lead to void moves (reflective boundary).
In our particular setting an agent starts in the bottom left corner and aims to learn the fastest way to reach the upper right corner.
For the states $s$ shown as blue squares in Fig.~\ref{fig:2D_case}a, the action proceeds according to the selected policy $\pi$. 
In the heated states (the magenta, orange and yellow squares in Fig.~\ref{fig:2D_case}a) the temperature (noise) affects the outcomes. 
Random offsets described by $W_t$ are added to the action selected from the policy $\pi$ for the state $s$. 
% If $T=1$, we assume that one random step is added in one of the four directions, i.e., $W_t$ take values $(1,0), (-1,0), (0,1),(0,-1)$. 
% No diagonal steps are allowed. 
% \sout{For example, if $T=2$, $W_t$ takes a value $(2, 0)$, but not $(2,1)$. $(1, 1)$ counts as two steps.} 
The actual effect of $W_t$ depends on the dimension of the gridworld and on a parameter $T$, called \textit{temperature} of the state.

% All moves are resolved in a single time frame, hence, an agent does not receive step penalty for extra random steps.
% The motion in the heated gridworld follows a recursive procedure that reads, for every time frame $t$:
% \begin{enumerate}
% % \vspace{-1mm}
%   \item At a position $s_t \in \mathbb Z^2$, sample $w_t$ from $W_t$ according to the temperature $T_t$.
%   \item Compute an action $a_t$ via $\pi(s_t)$.
%   \item Update the state via $s_{t+1} \leftarrow s_t + a_t + w_t$.
% \end{enumerate}
The motion in the studied heated gridworld follows a procedure that reads, for every time frame $t$:
\begin{enumerate}
  \item Compute an action $a_t$ using $\pi(s_t)$, set $\text{actions}=[a_t]$.
  \item For the position $s_t \in \mathbb Z^2$ take corresponding temperature $T_t$, sample $T_t$ values $w_{ti}$ and append them to the action list to yield the effective action list $[a_t, w_{t1}, ..., w_{tT}]$ .
  \item Update the state sequentially applying moves from the effective action list.
  Ordering of actions in this list matters due to possible interactions with obstacles and boundaries.
%   $$s_{t+1} \leftarrow s_t + a_t + w_{t1} + ... + w_{tT}$$
    % $$s_{t+1} \leftarrow s_t + a_t + w_t.$$
\end{enumerate}

An example of final position scattering is given in Fig. \ref{fig:heated_scatter}.
For all numerical experiments shown below, the tuning parameters were set as $\varepsilon$ = 0.1, $\gamma$ = 0.9.

\begin{figure}[H]
\centering
\includegraphics[width=\linewidth]{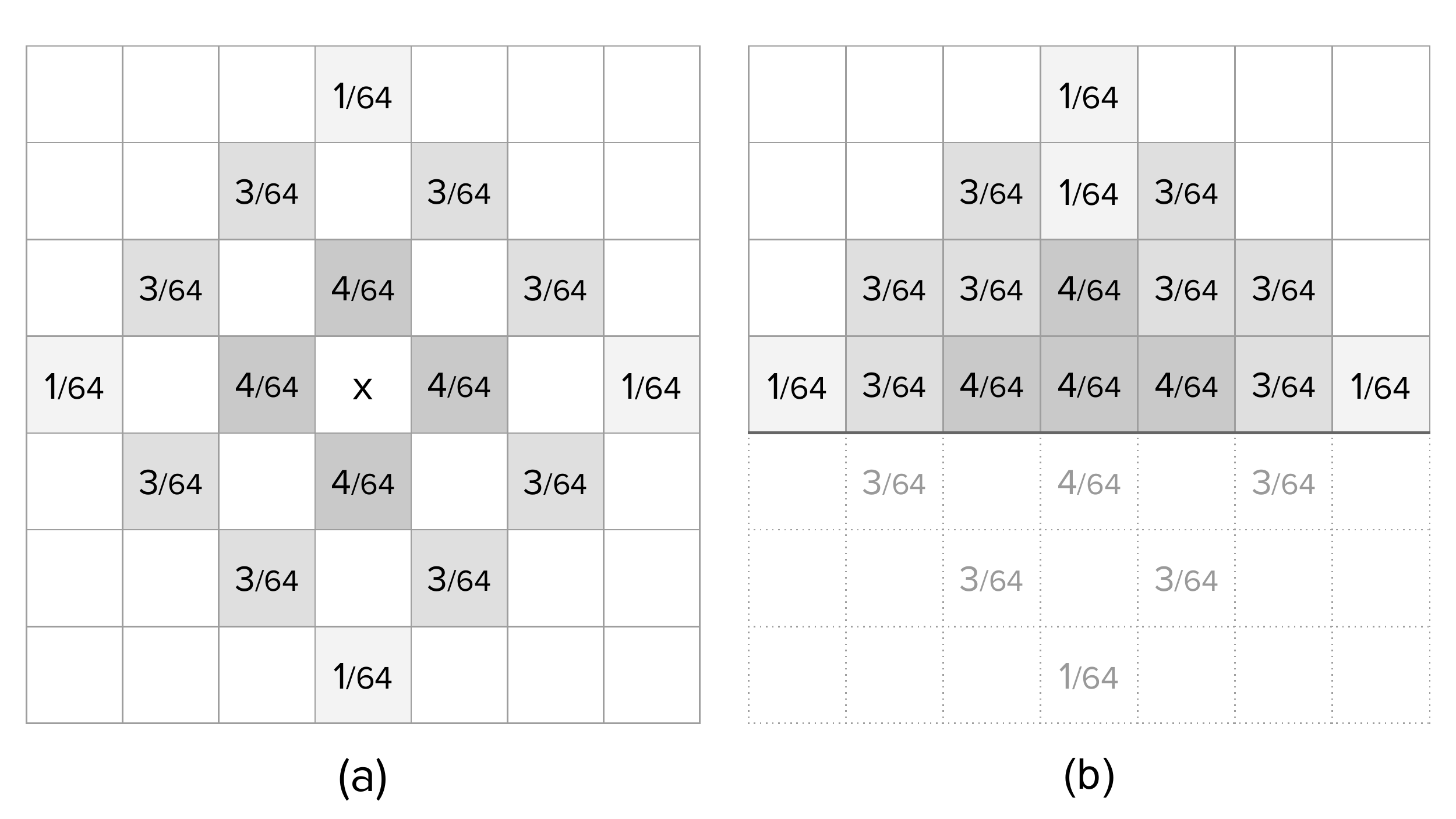}
\caption{Probability to finish in different positions for $T=3$ due to noise scattering, 'x' marks the starting position.
Left, (a): unrestricted; right, (b): near the world’s borders.}
\label{fig:heated_scatter}
\end{figure}

%All simulations were performed on a several supercomputer nodes with Intel(R) Xeon(R) Gold 6230 CPU @ 2.10GHz with 20 core(s) per socket and 4 sockets, 80 cores in total. A single run of 2D simulation with 1000 agents and 300K learning frames takes 15min on 48 cores. All code is done in basic python/numpy without third-party libraries and included in supplementary material.

\section{2D simulations}
\label{sec:2d-sim}

In this section, we consider an agent on a 10 $\times$ 10 grid shown in Fig. \ref{fig:2D_case} (b), $r_{tick}$ = -1 and $r_{target}$ = 100.
The heated region is placed in the lower right quarter of the grid. Temperature $T$ of this region is constant throughout the learning episode and is varied between the episodes from $T=0$ (symmetric setup) to $T=3$. The symmetric L-shaped obstacle leaves only two possible ways to reach the end tile from the start, either through deterministic part of medium or through the heated region.
The learning rate $\alpha$ is varied in the range [0.07, 0.09, 0.1, 0.2, 0.3, 0.4, ..., 0.9]. 
The quantity we aim to optimize is the mean first-passage time.

We mark agents as \textit{failed} if their time scores are higher than 500 timesteps cutoff, which is the case for trajectories closed in a loop for a long time. 
Regardless of $T$, this setup always has one option of a path 18 steps long with MFPT = 18 and zero standard deviation.
In the absence of temperature fluctuations ($T=0$), Q-learning converges to this optimum readily in 20K time steps for $\alpha>0.09$.
%and a bit longer for smaller values of $\alpha$. 
Once we introduce the noise, $T>0$, the number of failed agents changes (see Fig. \ref{fig:failed_agents}): it increases for high $\alpha\sim 0.8$ while for low $\alpha$ = 0.07 this number drops.

\begin{figure}[H]
\centering
\includegraphics{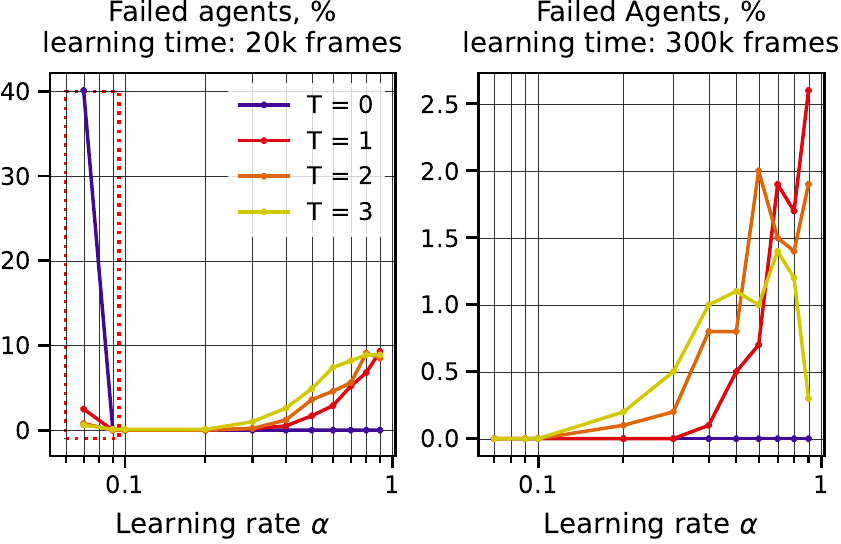}
\caption{Percent of agents failed to finish the game in 500 time frames for $10\times 10$ heated gridworld. For Q-learning and 1000 agents performance is measured for greedy behavior ($\varepsilon$ = 0) after learning time 20K and 300K time frames.
The increase in temperature drastically improves this statistics for $\alpha = 0.07$ (red dotted frame).}
\label{fig:failed_agents}
\end{figure}

\vspace{-2mm}

In the following we do not count the failed agents and only consider what happens with successful ones. The successful agents basically choose between two routes, the deterministic and through the heated region, depending on their learning rate, Fig. \ref{fig:heated_agents}. For $\alpha \sim 0.09$ one observes nearly 100\% convergence to the heated route, while for $\alpha > 0.4$ the majority selects the deterministic path after 300K played iterations. For higher $\alpha$ values the transit to the deterministic route occurs in shorter time. Scheduling of $\alpha$ that decreases its value with learning time thus leads to the population of agents staying in the heated area. Importantly, these changes occur for all tested algorithms similarly, as the path density plots show in Fig. 6. Only SARSA stands out being notably unstable. Remarkably, our findings do not show any advantage of double estimator in this task (see the comparison in Appendix \ref{additional}). In Fig. 5 one can clearly see that the presence of thermal noise increases the mean first-passage time for all $\alpha$, even for $\alpha$ values for which agents seem to operate well in heated areas. Higher learning rates produce a worse score with a bigger deviation after short learning (Fig. \ref{fig:2D_mfpt}, top row).
However, after learning for longer, they achieve a much better performance (Fig. \ref{fig:2D_mfpt}, bottom row) due to switching to another route.
%Under $\alpha \sim 0.1$, the agents continue to show result that is 2 time steps (11\%) worse than baseline at $T=0$.

% Higher learning rates produce a worse score with a larger spread of first-passage time distribution after short learning (see Fig. \ref{fig:2D_mfpt}, the upper part).
% After learning for longer, agents achieve a better performance (see Fig. \ref{fig:2D_mfpt}, the bottom part).
% However, mean scores are strictly worse than baseline one at $T=0$.

\begin{figure}[h]
    \centering
    \includegraphics{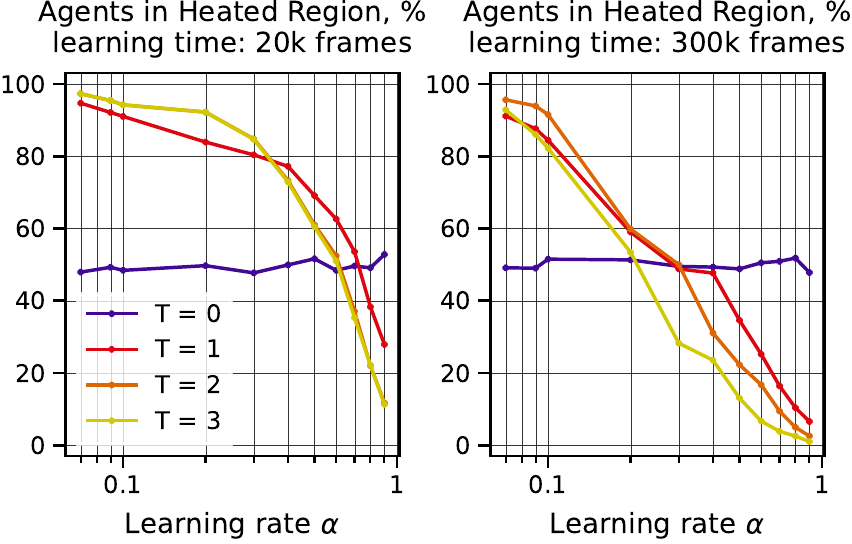}
    \caption{Percent of agents going through heated region for environments with different temperatures $T$ for $10\times 10$ heated 2D gridworld. Q-learning, 1000 agents, performance is measured for greedy behavior ($\varepsilon$ = 0.0) after learning time 20K and 300K time steps.}
    \label{fig:heated_agents}
\end{figure}

In our simulations the observed effect does not depend on the heated region location (Fig. \ref{fig:heated_locations}), the presence of obstacles and the value of discounting parameter $\gamma$. 
When several heated regions with different temperatures $T$ are placed in the gridworld, agents are divided between them proportionally to region temperature (Appendix \ref{additional}). The particular geometry with L-shaped wall allows us to demonstrate the effect quantitatively.

\begin{figure}[h]
\centering
 \begin{subfigure}{\linewidth}
 \centering
  \includegraphics[width=0.8\linewidth]{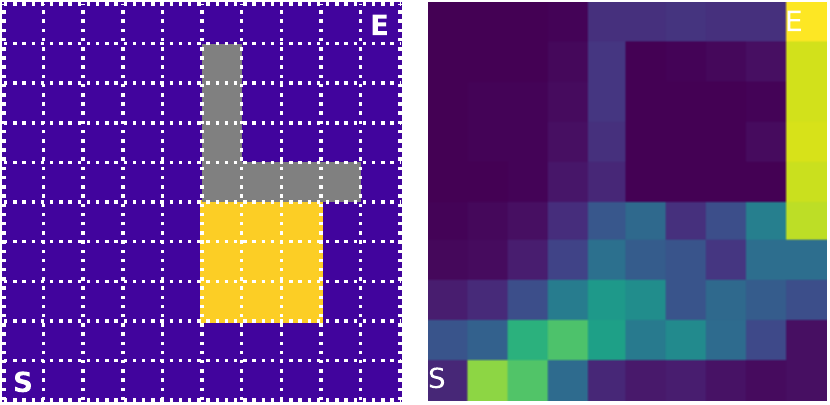}
  \label{fig:heatedgridworld_location1}
 \end{subfigure}
 \begin{subfigure}{\linewidth}
 \centering
  \includegraphics[width=0.8\linewidth]{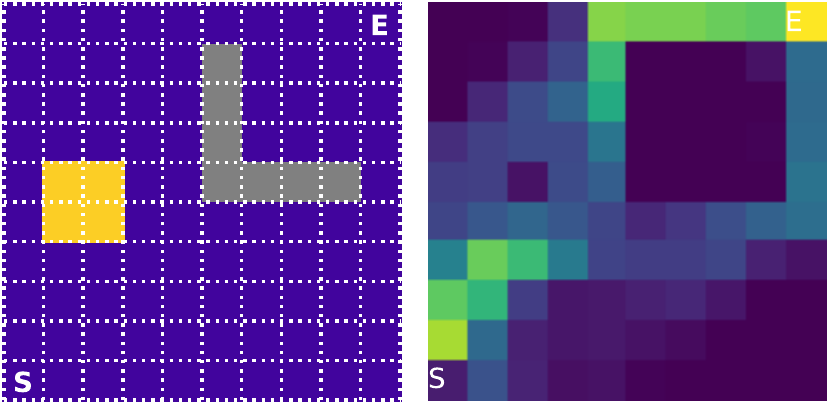}
  \label{fig:heatedgridworld_location2}
 \end{subfigure}
 \caption{Different locations of heated area. 100 agents, $\alpha = 0.1$, learning time is 50K time frames.}
\label{fig:heated_locations}
\end{figure}

\vspace{-30mm}

Our explanation of the effect is that the environmental noise boosts the exploratory behavior of an agent in some parts of the state space, therefore the policy tends to converge to regions with high temperature.

We found that transition to deterministic route when $T=3$ and $\alpha=0.1$ happens in 5M frames or 250K played episodes. Setting epsilon to 1.0 during the whole course of learning forms similar policy in 50K iterations which is equivalent to 100 played episodes (see Appendix \ref{additional}).

\vspace{10mm}

\begin{figure*}
 \begin{subfigure}{\linewidth}
  \includegraphics[width=0.99\linewidth]{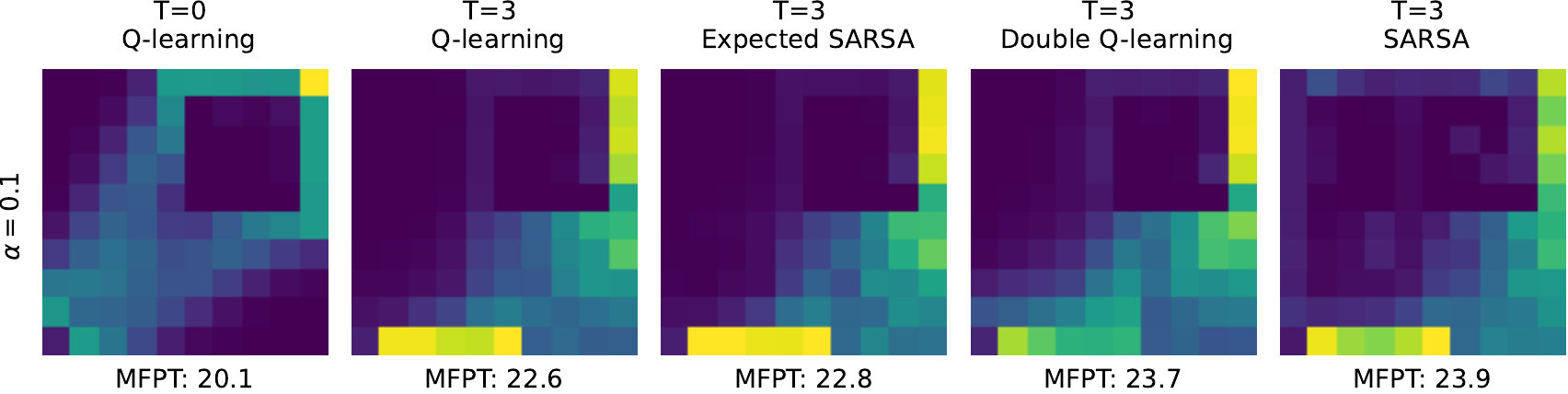}
  \label{fig:path_density_01}
 \end{subfigure}
 \begin{subfigure}{\linewidth}
  \includegraphics[width=0.99\linewidth]{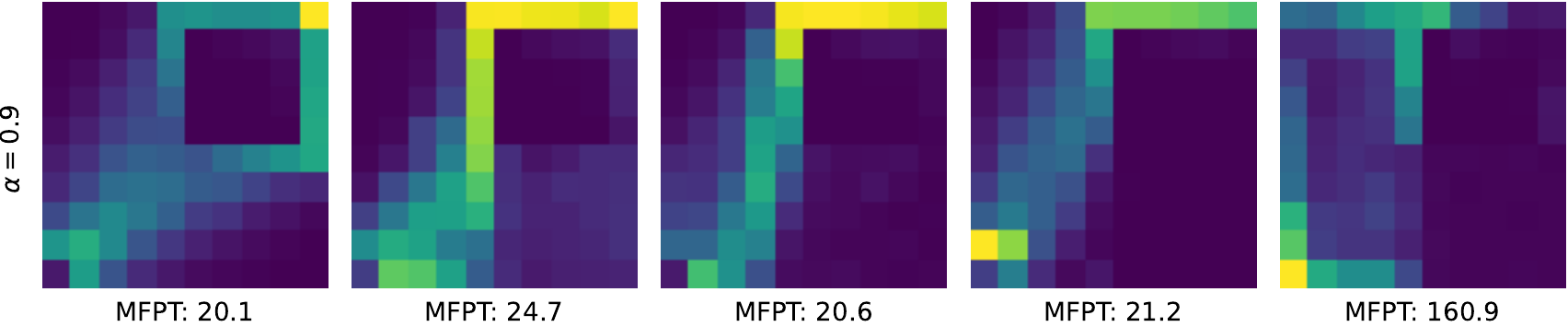}
  \label{fig:path_density_09}
 \end{subfigure}
 \caption{Path density changes due to the presence of a heated region for different learning algorithms. 500 agents. Learning time is 50K time frames in each case. Top row: $\alpha=0.1$, bottom row: $\alpha=0.9$.}
 \label{fig:path_density}
\end{figure*}

\begin{figure*}
\begin{subfigure}{\linewidth}
  \includegraphics[width=\linewidth]{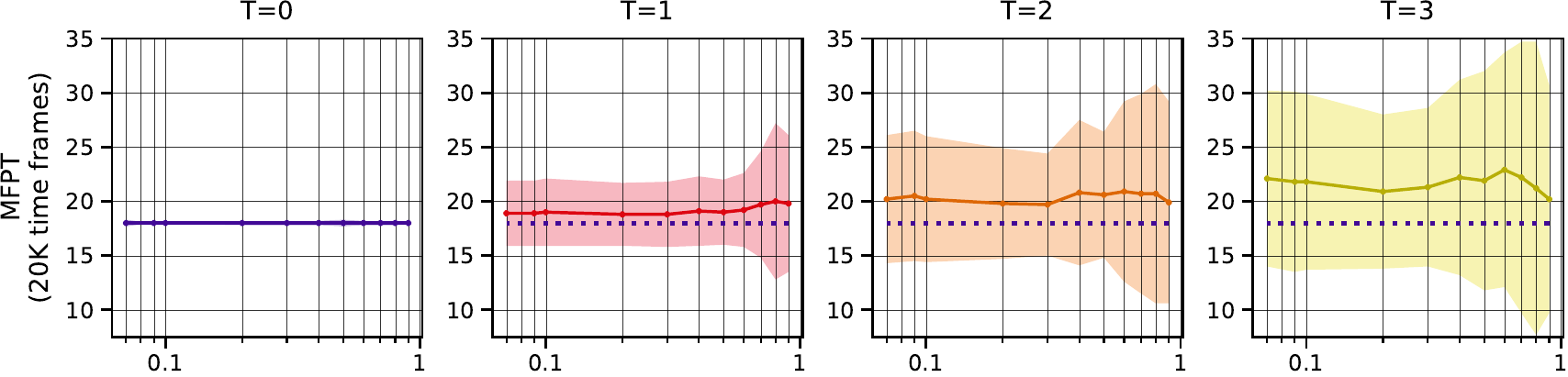} 
%   \caption{}
%   \label{fig:2D_mfpt_20k}
 \end{subfigure}
 \begin{subfigure}{\linewidth}
  \includegraphics[width=\linewidth]{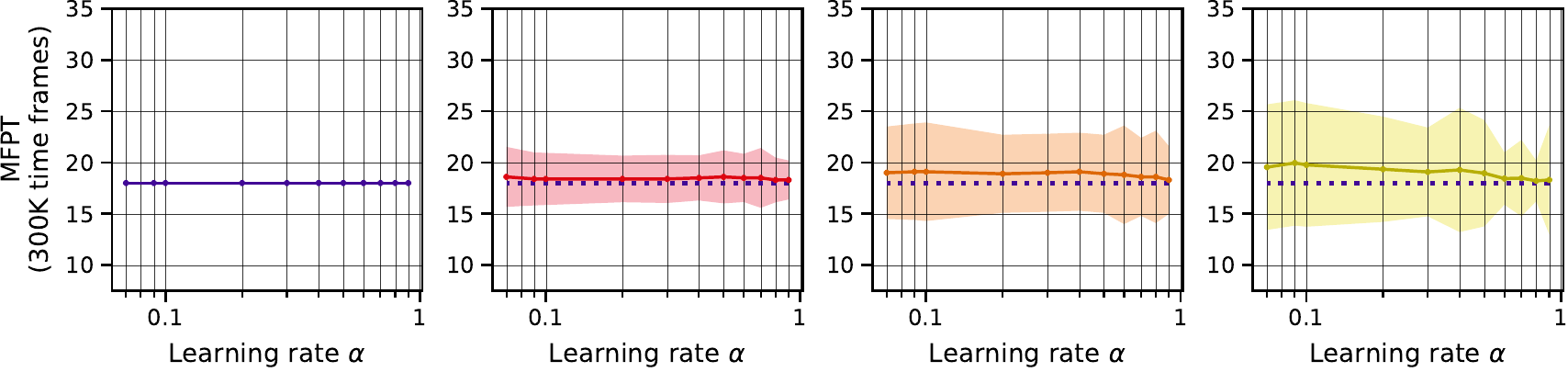}
%   \caption{}
%   \label{fig:2D_mfpt_300k}
  \end{subfigure}
 \caption{Mean first-passage time for environments with different $T$. Q-learning, 1000 agents, performance was measured for greedy behavior $\varepsilon$ = 0. Dashed line is baseline (MFPT at $T=0$). Top row: learning time = 20K time frames, bottom row: learning time = 300K time frames. The colored clouds depict one standard deviation of FPT distribution which comes from stochastic nature of the problem and depends on $T$ (not to be confused with measurement error).}
 \label{fig:2D_mfpt}
\end{figure*}

\section{1D simulations}
\label{sec:1d-sim}

\subsection{Interval with absorbing boundaries}

The 2D case shown in the previous section provided a fairly illustrative picture. However, the 1D case is easier to analyze and comprehend. We mainly studied a 1D gridworld where the agent starts in the center $x_0$ of an interval consisted of 41 states. Only two actions are available for the agent, namely, the moves to the left and to the right.
The reward comprises $-1$ for each time frame until the boundary is crossed, and $r_{\text{target}}$ is set to zero. The whole right-hand side of the gridworld $x > x_0$ has the temperature $T$.
1D case does not have interaction with borders and despite its simplicity still allows to produce a bias, as will be seen below.

We introduce the notations for policies shown in Fig. \ref{fig:1D_case}.
By $\pi_{L}$($\pi_R$) we denote the policy that admits a left (right) step, starting from the middle.
Analogously, to indicate a policy bias of two steps, we use a notation like $\pi_{RR}$, which means that from two states to the left of the middle the most common learned policy is to step right. Further policies are denoted following this token.

\begin{figure}[H]
  \centering
  \includegraphics[width=0.8\linewidth]{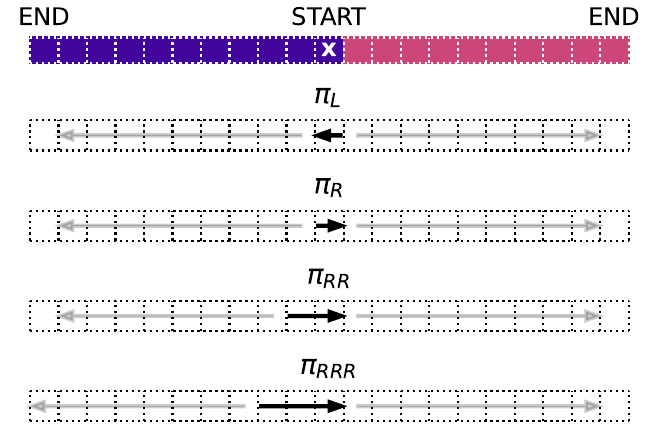}
  \caption{1D environment example and policies notation}
  \label{fig:1D_case}
\end{figure}

In simulations of this section the learning rate $\alpha =0.1$ was fixed. 
Overall, impact of $\alpha$ in 1D case is exactly the same to the one described in the previous section. High values $\alpha\sim0.5$ prevent algorithms from operating in stochastic media. 
It should be noted that notions of ``high'' and ``low'' $\alpha$ are relative and depend on average fluctuation scale in an environment. For instance, an interval with 5 states and $T=1$ has move length fluctuations which are 30\% of optimal path length. The low rate is then 0.01 and the high is 0.1. 

Policies of interest were tested by $10^6$ Monte Carlo (MC) runs. The best found in this way policy is denoted as $\pi^\star$, whereas $\pi_Q$ represents the most common policy of $10^4$ Q-learning agents. It was calculated by taken the most common action among the population of learning agents for every cell on the interval.

The difference between the best found policy $\pi^\star$ and the real agent's performance was observed at $\sim 1.5-2$ \% on average for all considered temperature levels, as Figure \ref{fig:1D_mfpt} shows. In contrast with 2D problem there is a little improvement in agents' scores as learning time passes.

\begin{figure}[H]
  \centering
  \includegraphics[width=1.0\linewidth]{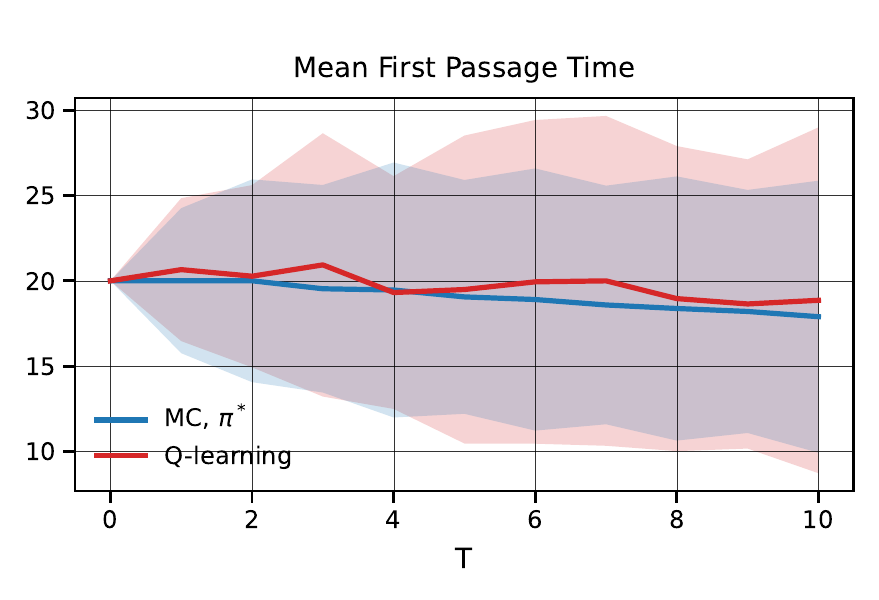}
  \caption{Q-learning performs slightly worse than the best found policy. MC: $10^6$ runs, Q-learning: $10^4$ agents, $\alpha =0.1$, learning time = 50K time frames. The lines show MFPT while the clouds come from the distribution of first passage times. Observe the increase of dispersion in first passage times when applying Q-learning.}
  \label{fig:1D_mfpt}
\end{figure}
 
Table \ref{table:1D} demonstrates that when the action "right" is optimal in position $x_0$ agents perform it in $x_0$ and $x_0-1$, i.e. instead of  $\pi_R$ they follow $\pi_{RR}$, instead of $\pi_{RR}$ they follow $\pi_{RRR}$ with an exception of weak noises ($T=1,2$).

\begin{table}[H]
  \caption{The most common policy $\pi_Q$, $10^4$~agents, $\alpha$~=~0.1, learning~time~=~50k time frames}
  \label{table:1D}
  \centering
  \begin{tabular}{|p{10pt} p{25pt} p{30pt} p{30pt} p{30pt}|} 
    \hline
      &  & \multicolumn{3}{c}{Right actions,\%} \\
    \cmidrule(r){3-5}
    T & $\pi_Q$ & $x_0$ & $x_0 - 1$ & $x_0 - 2$ \\
    \hline
    0 & $\pi_0$ & 50 & 0 & 0 \\
    1 & $\pi_R$ & \orange{99} & 0 & 0 \\
    2 & $\pi_R$ & \orange{97} & 0 & 0 \\
    3 & $\pi_{RR}$ & \green{100} & \red{75} & 0 \\
    4 & $\pi_{RR}$ & \green{100} & \red{91} & 0 \\ 
    5 & $\pi_{RR}$ & \green{100} & \red{96} & 1 \\ 
    6 & $\pi_{RR}$ & \green{100} & \red{100} & 1 \\ 
    7 & $\pi_{RR}$ & \green{100} & \red{100} & \red{50} \\ 
    8 & $\pi_{RR}$ & \green{100} & \red{100} & \red{41} \\ 
    9 & $\pi_{RRR}$ & \green{100} & \red{100} & \red{95} \\ 
    10 & $\pi_{RRR}$ & \green{100} & \green{100} & \red{80}\\
    \hline
  \end{tabular}
\end{table}

\subsection{Effect of drift in the heated region}

In order to strengthen the statement that the algorithms considered are prone to biases we construct the following example. We add a drift pushing the agent from the end of the interval in the heated part (see Fig. \ref{fig:1D_case_drift}). Its purpose is to gradually make the $\pi_R$ policy less profitable. The drift in our numerical experiment occurs only in the last right quarter of the interval (i.e. 25 per cent of the length). Its action on the agent is defined by a probability to make an extra move left.

\begin{figure}[H]
  \centering
  \includegraphics[width=0.8\linewidth]{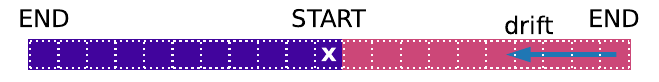}
  \caption{Example of 1D environment with a drift in some of the heated states}
  \label{fig:1D_case_drift}
\end{figure}

\begin{figure}[H]
 \begin{subfigure}{\linewidth}
  \includegraphics[width=\linewidth]{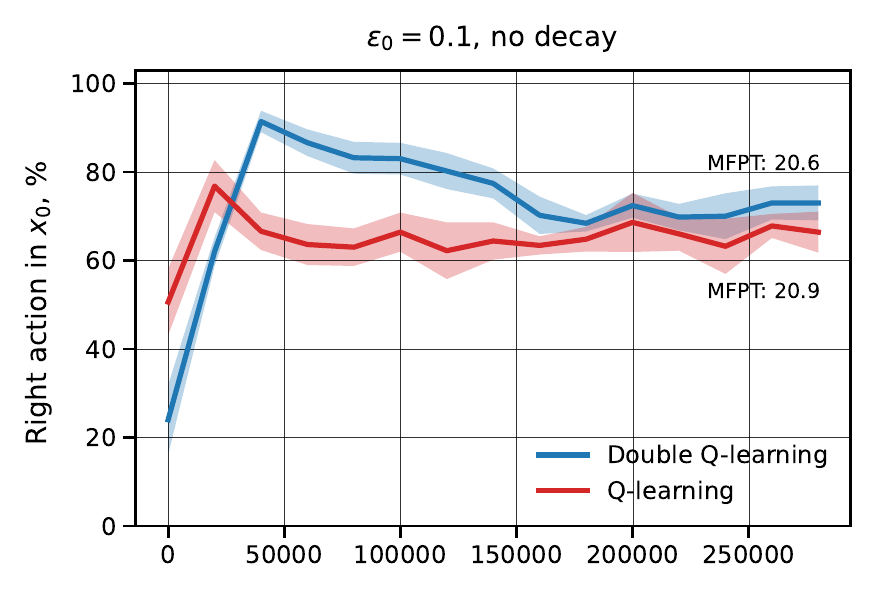}
 \end{subfigure}
 \begin{subfigure}{\linewidth}
  \includegraphics[width=\linewidth]{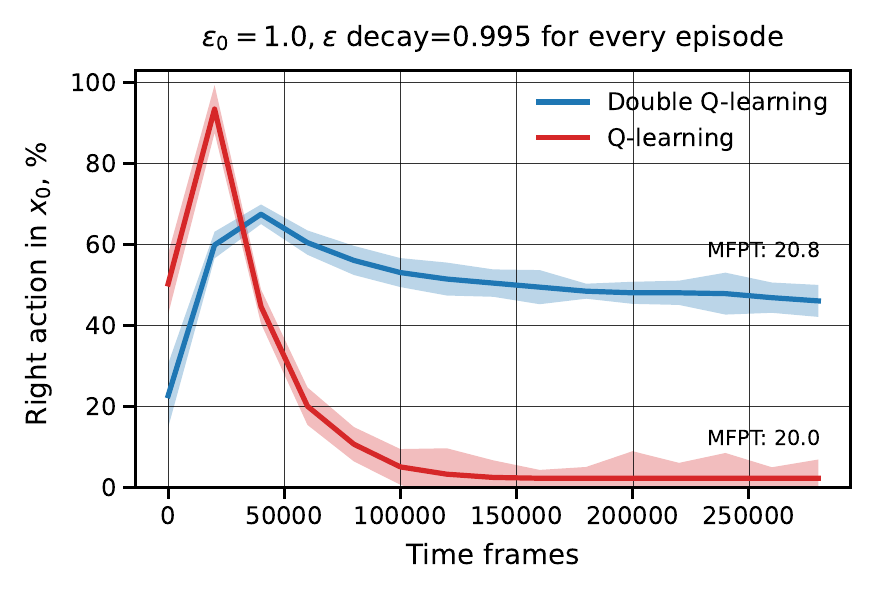}
 \end{subfigure}
 \caption{Simulation with drift value = 0.3, $T=3$, $\alpha=0.1, \gamma=0.9$, statistics is gathered over 5 runs with 100 agents. The right action in the center $x_0$ is non-optimal. Top: $\varepsilon=0.1$, biased policy is hard to overcome for both Q-learning and double Q-learning. Bottom: exponential decay of $\varepsilon$ enables Q-learning to achieve the optimal strategy, but has a moderate effect with Double Q-learning.}
 \label{fig:policy_trend}
\end{figure}
In the absence of thermal fluctuations ($T=0$) agents easily detect the drift for shift probability of $\sim0.1$ and select $\pi_L$ policy. Turning on even small $T=1$ effectively hides this non-optimality for majority of population.

Increase in fluctuations scale makes it possible to hide more intensive drift, as Table \ref{table:drift} shows. At $T=1$ only one agent out of ten is insensitive to the drift probability of 0.1 which yields approximately $9\%$ worse mean score. Similar $9\%$ decrease at $T=3$ is given by the drift value of 0.35, but this time it stays unnoticed by nearly 40\% of agents.

The numbers above are obtained after 50K time frames for Q-learning.
The top plot in Fig. \ref{fig:policy_trend} shows how a percent of non-optimal right actions in the center $x_0$ is changed for Q-learning and Double Q-learning through the learning process.
As one can see, there is no steady progress towards the optimal policy for both algorithms after playing 300K time frames.

First-passage properties of media affect the policy selection. Poorly trained agents are able to reach target in first stages of learning faster due to noise kicks and then they tend to stick to these policies.

As in 2D simulations, the proper scheduling of $\varepsilon$ for Q-learning (exponential decay starting from $\varepsilon_0=1.0$ in this case) significantly improves convergence to optimal MFPT value, Fig. \ref{fig:policy_trend}, bottom plot. Surprisingly, the scheduling has very moderate effect when applied with Double Q-learning, which is conventionally seen as a superior alternative of plain Q-learning for stochastic environments.

\begin{table}[H]
  \caption{Q-learning for the case with drift: $10^4$~agents, $\alpha$~=~0.1, learning~time~=~50k~time steps, difference between policies: MC $10^4$ runs. Right actions are suboptimal.}
  \label{table:drift}
  \setlength{\tabcolsep}{4pt}
  \centering
  \begin{tabular}{|p{10pt} p{25pt} p{50pt} p{80pt}|}
    \hline
    T & drift & $\displaystyle\frac{\pi_R - \pi_L}{\pi_L}$, \% & Right actions in $x_0$, \% \\
    \hline
 1 & 0.10 & 4.3 & 75 \\
   & 0.15 & 6.7 & 31 \\
   & 0.20 & 9.5 & 10 \\ 
    \hline
 2 & 0.10 & 3.7 & 67 \\
   & 0.15 & 5.9 & 36 \\ 
   & 0.20 & 8.3 & 16 \\
    \hline
 3 & 0.10 & 0.5 & 99 \\ 
   & 0.15 & 2.2 & 96 \\ 
   & 0.20 & 3.7 & 94 \\ 
   & 0.25 & 5.9 & 86 \\ 
   & 0.30 & 7.5 & 68 \\
   & 0.35 & 9.2 & 42 \\
   & 0.40 & 12.0 & 25 \\
%   & 0.45 & 14.2 & 8 \\ [0.7ex]
    \hline
\end{tabular}

\end{table}

% --------- sentence about double Q moved to 2D case:
% Remarkably, our findings do not show a substantial advantage of double Q-leaning, besides more smooth transition between policies during learning time, Figure \ref{fig:policy_trend}.

%The results probably could be extended to algorithms with update rule similar to TD(0).

\section{Conclusions}

The application of machine learning algorithms to real physical systems have to be tested and vetted by using toy models in order to understand and anticipate possible biases. In this paper we find that four well-established tabular reinforcement learning algorithms show bias in terms of producing suboptimal solutions for the problem of fastest boundary crossing in gridworlds with state-dependent noise. We name this type of gridworlds as heated gridworlds. The state-dependent noise affecting the work of algorithms can occur for different physical reasons starting from uneven temperature distribution or concentration variations in the case of atmosphere pollution to long-lived current patterns in the water or the atmosphere. For 1D and 2D heated gridworlds we see a pronounced bias for  Q-learning,  SARSA,  Expected  SARSA  and  Double  Q-learning: High learning rate prevents the algorithm from operating in stochastic media while for small enough $\alpha$ agents tend to go through noisy regions, even when these policies are suboptimal. 

%Loss in performance between optimal and selected policies was found to reach 7\% for majority of agents. <----- certain values are env-dependent and are not so important, better to omit

We clearly see that the methods developed in recent years to tackle the case of noisy rewards (e.g. Double Q-learning) do not necessarily offer the same benefits for learning in the case of stochastic transition dynamics. Our work is a sandbox example which could be useful for those who looks into new applications of temporal-difference algorithms primarily in physics, navigation and trading or generally study capabilities of these optimisation tools. We expect that similar effects could be found in other environments with unevenly distributed noise.

% Certain important observations were made which should, in our view, be taken into account when using reinforcement learning in similar physical problems, namely:
% \begin{itemize}
% \item High learning rate prevents the algorithm from operating in stochastic media.
% \item For small enough $\alpha$ agents tend to go through noisy regions, even when these policies are suboptimal.
% \item Loss in performance between optimal and selected policies can reach 7\% for majority of agents.
% \item Biased policies were observed hard to overcome.
% \end{itemize}

\appendix

\section{Algorithms}\label{algorithms}
\begin{algorithm}
% \KwResult{Write here the result }
\DontPrintSemicolon
\SetAlgoNoLine
 Initialize $Q(s, a)$, $S$ \;
 \Repeat{\upshape done}{
  Choose $A = \argmax_a Q(S, a)$ or with $\varepsilon$ probability choose random $a$ \;
  Take action $A$, observe $R$ and $S'$ \;
  $Q(S, A) \leftarrow Q(S, A) + \alpha \ \big( r + \gamma \ \max_a Q(S', a) - Q(S, A)\big)$ \;
   $S \leftarrow S'$
 }
 \caption{Q-learning with $\varepsilon$-greedy policy \cite{watkins1992q}}
 \label{algo:q-learning}
\end{algorithm}
% \vspace{-0.7cm}
\begin{algorithm}
% \KwResult{Write here the result }
\DontPrintSemicolon
\SetAlgoNoLine
 Initialize $Q(s, a)$, $S$\;
 Choose $A$ from $\varepsilon$-greedy policy\;
 \Repeat{\upshape done}{
  Take action $A$, observe $R$ and $S'$ \;
  Choose $A'$ from $\varepsilon$-greedy policy \;
  $Q(S, A) \leftarrow Q(S, A) + \alpha \ \big(R + \gamma \ Q(S', A') - Q(S, A)\big)$ \;
   $S \leftarrow S'$, $A \leftarrow A'$
 }
 \caption{SARSA \cite{rummery1994sarsa}}
 \label{algo:sarsa}
\end{algorithm}
% \vspace{-0.7cm}
\begin{algorithm}
% \KwResult{Write here the result }
\DontPrintSemicolon
\SetAlgoNoLine
 Initialize $Q(s, a)$, $S$ \;
 \Repeat{\upshape done}{
  Choose $A$ from $\varepsilon$-greedy policy \;
  Take action $A$, observe $R$ and $S'$ \;
  $Q(S, A) \leftarrow Q(S, A) + \alpha \ \big( R + \gamma \ \mathbb{E}_{\pi}\big[ Q(S', a)\big] - Q(S, A)\big)$ \;
   $S \leftarrow S'$
 }
 \caption{Expected SARSA \cite{john1994expected_sarsa}}
 \label{algo:expected_sarsa}
\end{algorithm}
\begin{algorithm}
% \KwResult{Write here the result }
\DontPrintSemicolon
\SetAlgoNoLine
 Initialize $Q^A(s, a)$, $Q^B(s, a)$, $S$ \;
 \Repeat{\upshape done}{
  Choose $A$ based on $Q^A(S,\cdot)$ and $Q^B(S, \cdot)$, observe $R$ and $S'$ \;
  Choose random UPDATE(A) or UPDATE(B) \;
  \uIf{\upshape UPDATE(A)}{
     Define $a^\star = \argmax_a Q^A(S, a)$\;
     $Q^A(S, A) \leftarrow Q^A(S, A) + \alpha \ \big( r + \gamma \ Q^B(S', a^\star) - Q^A(S, A)\big)$ \;
   }
   \uElseIf{\upshape UPDATE(B)}{
     Define $b^\star = \argmax_a Q^B(S, a)$\;
     $Q^B(S, A) \leftarrow Q^B(S, A) + \alpha \ \big( r + \gamma \ Q^A(S', b^\star) - Q^B(S, A)\big)$ \;
   }
   $S \leftarrow S'$
 }
 \caption{Double Q-learning \cite{hasselt2010double}}
 \label{algo:double_learning}
\end{algorithm}

\vspace{\parskip}

\section{Theoretical analysis}\label{theory}

In this appendix we provide some theoretical analysis of the case considered in Section \ref{sec:1d-sim}.
We show bounds on survival probabilities and finally show well-posedness of the studied gridworld boundary crossing problem. The following notation will be used: $\overline{0,N}$ means the set $\{0, 1, \dots, N\}$.
We fix an ambient probability space $(\Omega, \Sigma, \mathbb P)$ and consider (cf. \eqref{eqn:gridworld-env}):
\begin{equation}
\label{eqn:statement}
    \begin{array}{l}
        \tilde{S}_{t + 1} = \tilde{S}_t + \pi(\tilde{S}_t) + W_t(\tilde{S}_t), \\
        S_{t + 1} = \tilde{S}_{t + 1}, \lvert\tilde{S}_{t + 1}\rvert \leq \zeta \ \land S_t \in \left[-\zeta, \zeta\right], \\
        S_{t + 1} = s_*, \text{ otherwise}, \\
        W_t(k) = 2\eta_t(k) - w_k, \ \eta_t(k) \sim \mathcal{B}\left(w_k, \frac{1}{2}\right), \ w_k \geq w > 1, \\
        \forall k \in \mathbb{Z} \ : \ \pi(k) \in \left\{-1, 0, 1 \right\}, \\
        S_0 = \tilde{s}_0 \in \left[-\zeta, \zeta\right] \cap \mathbb{Z},
    \end{array}
\end{equation}
where $\mathcal B(a, b)$ denotes the binomial distribution with parameters $a, b$ and the policy comprises of the choices to go either left or right (we allow zero ``noise-induced'' move here only for convenience).
% In this section, we will show bounds on the survival probability when the temperature effects dominate the 
%Let $\Omega$ be the sample space of this stochastic model.
Let $\bar{\mathcal S}$ denote $[-\zeta, \zeta]\cap\mathbb{Z}$ and let $\mathcal S$ denote $\bar{\mathcal S}\cup\{x_*\}$.
Note that the state-space equation for $S$ can be rewritten as (cf. \eqref{eqn:gridworld-env-alt})
\begin{equation}
%  \left\{
  \begin{array}{ll}
	S_{t + 1} = S_t + \pi(S_t) + W_t(S_t), & S_t \in \left[-\zeta, \zeta\right] \land \\
	& \lvert S_t + \pi(S_t) + W_t(S_t) \rvert\\
	& \leq \zeta \\
	S_{t + 1} = s_*, & \text{otherwise}
  \end{array}
%  \right\}.
\end{equation}

Therefore, $S$ and $\tilde{S}$ represent mutually ``independent'' dynamical systems that can be considered on their own. From this point onward the dynamical systems corresponding to $S$ and $\tilde{S}$ will be referred to as ``the absorbing system'' and ``the free system'' accordingly.
There is, however, a certain connection between the trajectories of the two systems. Since both systems by definition have identical noise and control, the two of them also share the same probability space.
Another notable connection is that the trajectories of the two systems coincide up until the first passage beyond $\zeta$ or $-\zeta$.
Both systems are evidently Markov decision processes with stationary control.
And since the absorbing system has a finite number of states, one can construct its probability transition matrix under the assumption that $\pi$ is fixed.
Obviously, the absorbing system only has $2r + 2$ states, $2r + 1$ of which are $-\zeta, 1 - \zeta, \dots, \zeta$ and the remaining one is $s_*$.
To construct the probability transition matrix $G$ we first need to order the states.
Let the order be $\left(-\zeta, \dots, \zeta, *\right)$, and thus let
$$
G^t = \left[g^t_{-\zeta}, \dots, g^t_{\zeta}, g^t_*\right]
$$
Now, let $\bar{G} = \left[\bar{g}^t_{-\zeta}, \dots, \bar{g}^t_{\zeta}\right]$ be the top left square submatrix of $G$ of size $2\zeta + 1$.
% , and thus similarly let 
% $$
% \bar{G}^t = \left[\bar{g}^t_{-\zeta}, \dots, \bar{g}^t_{\zeta}\right]
% $$

We have:
\begin{lemma}
\label{IteratedSubmatrixna}
The matrix $\bar{G}^t$ is the top left submatrix of $G^t$.
\end{lemma}

\begin{proof}
Since it is impossible to escape the absorbing state $s_*$, the following identity holds:
$$
\forall k \in \mathbb{N} \ : \ g^k_* = (0 \ 0 \ \dots \ 0 \ 1)^\top
$$

The lemma can be proven by simple mathematical induction.
Namely,
\begin{description}
    \item[Induction base:] 
\[G = 
\left[\begin{array}{@{}c|c@{}}
  \bar{G}
  &  \begin{matrix}
  0 \\
  0 \\
  \dots \\
  0 \\
  \end{matrix}\\
\hline
  \dots &
  1
\end{array}\right]
\]
    \item[Induction step:]
\[ 
\begin{array}{l}
	G G^t = 
	\left[\begin{array}{@{}c|c@{}}
	  \bar{G}
	  &  \begin{matrix}
	  0 \\
	  0 \\
	  \dots \\
	  0 \\
	  \end{matrix}\\
	\hline
	  \dots &
	  1
	\end{array}\right] \left[\begin{array}{@{}c|c@{}}
	  \bar{G}^t
	  &  \begin{matrix}
	  0 \\
	  0 \\
	  \dots \\
	  0 \\
	  \end{matrix}\\
	\hline
	  \dots &
	  1
	\end{array}\right] = \\
		\left[\begin{array}{@{}c|c@{}}
		  \bar{G}\bar{G}^t + 0^{(2r + 1) \times (2r + 1)}
		  &  \bar{G}0^{2r + 1} + 1\cdot0^{2r + 1}\\
		\hline
		  \dots &
		  0 + 1 \cdot 1
		\end{array}\right]
	\end{array}
\]
\end{description}

\end{proof}

Let $P^s_0$ be the initial probability distribution vector under the assumption that $S_0 = s$.
Structurally, the latter would be a binary vector with a single entry equal $1$ at the position that corresponds to the chosen initial state.
Therefore,
$$
G^tP^s_0 = g^t_s
$$
%This reveals the nature of $g^t_s$.
Evidently, $g^t_s$ is the distribution vector at step the $t$ under the assumption that $s$ was the initial state.

Similarly let us define $\bar{P}^s_0$ as $P^s_0$ without the last element.
The analogy is clear:
$$
\bar{G}^t\bar{P}^s_0 = \bar{g}^t_s.
$$

Let $H_t \subseteq \Sigma$ be the event that at step $t$ the absorbing system is still not in state $s_*$.
Obviously, $\PP{H_t}$ is the survival probability.
\begin{lemma}
\label{lem:SurvivalEssencena} If $s_0 = s$, then $P^s_t = \lVert \bar{g}^t_s \rVert_1$.
\end{lemma}

\begin{proof}
$\bar{g}^t_s$ is a subvector of $g^t_s$.
The only element of $g^t_s$ that $\bar{g}^t_s$ does not include is the one that corresponds to the absorbing state $s_*$.
Therefore, $\bar{g}^t_s$ comprises probabilities of being in states $-\zeta, \dots, \zeta$ at step $t$ accordingly.
Since $\lVert\bar{g}^t_x\rVert_1$ is the sum of absolute values of elements of $\bar{g}^t_x$, evidently $\lVert\bar{g}^t_x\rVert_1$ is the probability of being in any state other than $s_*$ by step $t$ or, equivalently,
$$
\PP{H_t} = \lVert \bar{g}^t_s \rVert_1.
$$
\end{proof}

We have the following upper bound on the survival probability:
\begin{lemma}
\label{lem:SurvivalUpperBoundna}
It holds that $\PP{H_\tau} \leq 1 - 2^{-w\tau}$, where $\tau = \floor{\frac{\zeta}{w - 1}} + 1$.
\end{lemma}

\begin{proof}
Without loss of generality let us assume that $s_0 \geq 0$. 

Let $K_t = \bigcap\limits_{i = 0}^{t - 1}\left\{\omega \in \Omega \ | \ W_i(\tilde{S}_i)[\omega] = w\right\}$.
The probability of each event within the latter intersection is no less than ${w \choose w} (\frac{1}{2})^w(\frac{1}{2})^0$, and since the events are independent, it holds that 
\begin{equation}
\label{eqn:EdgeProba}
    \PP{K_t} \geq 2^{-wt}.
\end{equation}

We denote the event that the state of the environment is outside $\mathcal S$ at $t$ as $E_t := \left\{\omega \in \Omega \ | \ \lvert\tilde{S}_t[\omega]\rvert > \zeta \right\}$.
According to (\ref{eqn:statement}), this event obviously implies that the agent has not survived by $t$, therefore, 
\begin{equation}
\label{eqn:implication1}
  E_t \subset \bar{H}_t.  
\end{equation}

For each elementary outcome from $K_t$, the following identity trivially follows from definition of $K_t$:
$$
\forall \omega \in K_t \ : \ \tilde{S}_t(\omega) = \sum_{i = 0}^{t - 1}\pi(s_i) + \sum_{i = 0}^{t - 1}w_i + s_0.
$$
According to (\ref{eqn:statement}), we have: 
$$
\pi(s_i) \geq -1, \ i = \overline{0,t - 1}.
$$

Therefore,
$$
\sum_{i = 0}^{t - 1}\pi(s_i) + \sum_{i = 0}^{t - 1}w + s_0 \geq -t + tw + s_0.
$$
thus
$$
\forall \omega \in K_t \ : \ \lvert\tilde{S}_(\omega) \rvert \geq t(w - 1) + s_0.
$$
Since the above holds for arbitrary $t \geq 1$, then it should also hold for $\tau(s): = \floor{\frac{\zeta - s}{w - 1}} + 1$ as long as $s \in \left[0, \zeta\right]$.
Evidently, $\tau(s)$ has the following important property
$$
\tau(s) > \frac{\zeta - s}{w - 1}.
$$
Therefore,
$$
\tau(s)(w - 1) + s > \zeta.
$$
This in turn means
$$
\forall \omega \in K_{\tau(s_0)} \ : \ \lvert\tilde{S}_{\tau(s_0)}(\omega)\rvert \geq \tau(s_0)(w - 1) + s_0 > \zeta,
$$
whence
$$
\forall \omega \in K_{\tau(s_0)} \ : \ \lvert\tilde{S}_{\tau(s_0)}(\omega)\rvert > \zeta.
$$
Notably, this implies that 
\begin{equation}
\label{eqn:implication2}
K_{\tau(s_0)} \subset E_{\tau(s_0)}    
\end{equation}
and, therefore, according to (\ref{eqn:EdgeProba}), (\ref{eqn:implication1}) and (\ref{eqn:implication2}), we have:
$$
K_{\tau(s_0)} \subset \bar{H}_{\tau(s_0)}
$$
and
$$
\PP{ \bar{S}_{\tau(x_0)} } \geq \PP{K_{\tau(x_0)} } \geq 2^{-w\tau(s_0)}.
$$
Finally, 
\begin{equation}
\label{eqn:finally}
\PP{\bar{S}_{\tau(x_0)}} \geq 2^{-w\tau(s_0)}.
\end{equation}
Let $\tau$ (with no argument) denote $\tau(0)$.
It can be easily observed from the definition of $\tau(s)$ that
$$
\forall s_0 \in [0, \zeta] \ : \ \tau(s_0) \leq \tau.
$$
This has two important implications.
Firstly,
\begin{equation}
\label{eqn:implication3}
  \forall s_0 \in [0, \zeta]\cap\mathbb{Z} \ : \ \bar{H}_{\tau(x_0)} \subset  \bar{H}_\tau  
\end{equation}
since obviously $\bar{H}_t$ implies $\bar{H}_{t + 1}$.
And secondly
\begin{equation}
\label{eqn:WeakerInequality}
    \forall s_0 \in [0, \zeta]\cap\mathbb{Z} \: \ 2^{-w\tau(x_0)} \geq 2^{-w\tau}.
\end{equation}
Combining (\ref{eqn:finally}), (\ref{eqn:implication3}) and (\ref{eqn:WeakerInequality}), we can determine that the following holds under the assumption that the initial state $s_0$ is non-negative
$$
\PP{\bar{H}_\tau} \geq 2^{-w\tau}, \ \tau = \floor{\frac{\zeta}{w - 1}} + 1
$$
or, equivalently,
$$
\PP{\bar{H}_\tau} \leq 1 - 2^{-w\tau}, \ \tau = \floor{\frac{\zeta}{w - 1}} + 1.
$$
The same can be proven for non-positive values of $s_0$ in a similar fashion.
\end{proof}

%\begin{prf}
%Let $K_t$ denote the event that the noise $W_i$ was equal to $w$ for $i=\overline{0, t - 1}$.
%Then,
%$$
%    \PP{K_t} \geq 2^{-wt}.
%$$
%Let $E_t$ denote the event that the state of the free system is outside $\mathcal S$ at $t$.
%Then 
%$$
%  E_t \subset \bar{H}_t,   
%$$
%where $\bar{H}_t$ is the complement of $H_t$.
%
%Observe that
%$$
%\forall \omega \in K_t \ : \ \lvert\tilde{S}_t[\omega]\rvert \geq t(w - 1) + s_0.
%$$
%
%Let $\tau(s) := \floor{\frac{\zeta - s}{w - 1}} + 1$.
%Then 
%$$
%\tau(s)(w - 1) + s > \zeta
%$$
%and
%$$
%\forall \omega \in K_{\tau(s_0)} \ : \ \lvert\tilde{S}_{\tau(s_0)}(\omega)\rvert \geq \tau(s_0)(w - 1) + s_0 > \zeta
%$$
%thus
%$$
% K_{\tau(s_0)} \subset E_{\tau(s_0)} \subset \bar{H}_{\tau(s_0)}
%$$
%and
%$$
%P\left(\bar{H}_{\tau(s_0)}\right) \geq P\left(K_{\tau(s_0)}\right) \geq 2^{-w\tau(s_0)}.
%$$
%It can be easily shown that the above holds for $\tau(0)$ regardless of $s_0$.
%\end{prf}

We can furthermore bound the survival probability as follows:
\begin{lemma}
There is a $C > 0$ such that the survival probability satisfies the relation
$$
 \forall \sigma > (1 - 2^{-w\tau})^{\frac{1}{\tau}} \ : \  \PP{H_t} \leq C \sigma^t.
$$
\end{lemma}

\begin{proof}
By Lemma \ref{lem:SurvivalEssencena} and Lemma \ref{lem:SurvivalUpperBoundna} we have
$$
\forall s \in [-\zeta, \zeta] \cap\mathbb{Z} \ : \ \lVert \bar{g}^\tau_s \rVert_1 \leq 1 - 2^{-w \tau},
$$
whence 
\begin{equation}
\label{eqn:NormRelation}
    \lVert \bar{G}^\tau \rVert_1^{\frac{1}{\tau}} = \left(\max\limits_s \lVert \bar{g}^\tau_s \rVert_1\right)^{\frac{1}{\tau}} \leq (1 - 2^{-w\tau})^{\frac{1}{\tau}}.
\end{equation}

%It is known \cite{sr_upper_bound} that for any consistent norm $\lVert \cdot\rVert$ and for any matrix $Q$ the following relation holds for its spectral radius $\rho$:
%$$
%\forall n \geq 1 \ : \ \rho\left(Q\right) \leq \lVert Q^n\rVert^\frac{1}{n}
%$$
%thus according to (\ref{eqn:NormRelation})
%$$
%\rho\left(\bar{A}\right) \leq (1 - 2^{-q\tau})^{\frac{1}{\tau}}
%$$
%It is known \cite{sr_implication} that for any matrix $Q$ the following relation holds for any norm $\lVert\cdot\rVert$:
%$$
%\forall w \ \forall \sigma > \rho\left(Q\right) \ \exists \ C(w) > 0 \ \forall t \geq 0 \ : \  \lVert Q^tw\rVert \leq C(w)\sigma^t
%$$
%therefore
Observe that
\begin{align}
& \forall s_0 \in [-\zeta, \zeta]\cup\mathbb{Z} \ \forall \sigma > (1 - 2^{-w\tau})^{\frac{1}{\tau}} \ \exists \ C(s_0) > 0 \ \forall t \geq 0 \ : \  \nonumber \\ 
& \lVert\bar{G}^t\bar{P}^{s_0}_0\rVert_1 = \lVert\bar{g}^t_{s_0}\rVert_1 = \PP{H_t} \leq C(x_0)\sigma^t \nonumber
\end{align}

Let $C$ (with no argument) denote $\max\limits_{s \in [-\zeta, \zeta]\cap\mathbb{Z}} C(s)$, whence
\begin{equation}
\label{eqn:GeneralFactor}
\begin{aligned}
	& \forall s_0 \in [-\zeta, \zeta]\cup\mathbb{Z} \ \forall \sigma > (1 - 2^{-w\tau})^{\frac{1}{\tau}} \ \forall t \geq 0 \ : \\
	& \lVert \bar{G}^t P^{s_0}_0 \rVert_1 \leq C\sigma^t.
\end{aligned}
\end{equation}

Earlier, we already established that $\bar{G}^t P^s_0 = \bar{g}^t_s$.
Therefore, by  Lemma \ref{lem:SurvivalEssencena} we have:
\begin{equation}
\label{eqn:SurvivalAndColumns}
\lVert\bar{G}^t \bar P^{s_0}_0 \rVert_1 = \lVert\bar{g}^t_{s_0}\rVert_1 = \PP{H_t}.    
\end{equation}

Combining (\ref{eqn:GeneralFactor}) and (\ref{eqn:SurvivalAndColumns}), we obtain
$$
\exists \ C > 0 \ \forall \sigma > (1 - 2^{-w\tau})^{\frac{1}{\tau}} \ \forall t \in \mathbb{N}\cup\{0\} \ : \  \PP{H_t} \leq C\sigma^t,
$$
where quantifying over $s_0$ was dropped, because according to (\ref{eqn:statement}) $H_t$ is independent of $s_0$.
\end{proof}

%\begin{prf}
%By Lemma \ref{lem:SurvivalEssencena} and Lemma \ref{lem:SurvivalUpperBoundna} we have 
%$$
%    \lVert \bar{G}^\tau \rVert_1^{\frac{1}{\tau}} = \left(\max\limits_s \lVert \bar{g}^\tau_s \rVert_1\right)^{\frac{1}{\tau}} \leq (1 - 2^{-w\tau})^{\frac{1}{\tau}}.
%$$
%%Since $\forall n \geq 1 \ : \ \rho\left(Q\right) \leq \lVert Q^n\rVert^\frac{1}{n}$, we have
%%$$
%%\rho\left(\bar{A}\right) \leq (1 - 2^{-q\tau})^{\frac{1}{\tau}}
%%$$
%%It is known \cite{sr_implication} that for any matrix $Q$ the following relation holds for any norm $\lVert\cdot\rVert$:
%%$$
%%\forall w \ \forall \sigma > \rho\left(Q\right) \ \exists \ C(w) > 0 \ \forall t \geq 0 \ : \  \lVert Q^tw\rVert \leq C(w)\sigma^t
%%$$
%Therefore, we have
%\begin{align}
%& \forall s_0 \in [-\zeta, \zeta]\cup\mathbb{Z} \ \forall \sigma > (1 - 2^{-w\tau})^{\frac{1}{\tau}} \ \exists \ C(s_0) > 0 \ \forall t \geq 0 \ : \  \nonumber \\ 
%& \lVert\bar{G}^t\bar{P}^{s_0}_0\rVert_1 = \lVert\bar{g}^t_{s_0}\rVert_1 = \PP{H_t} \leq C(x_0)\sigma^t \nonumber
%\end{align}
%which gives the result.
%\end{prf}

\begin{corollary}
There are also $C > 0, \sigma \in (0, 1)$ such that $\PP{H_t} \leq C\sigma^t$.
\end{corollary}

Now, let us formulate the reward function for the considered boundary crossing problem as follows (cf. \eqref{eqn:reward-gridworld}):
\[
\begin{array}{ll}
     r(s_t, \pi(s_t), w_t) =  \\
     \mathbb I_{\{s_*\}}(s_{t + 1})\mathbb I_{[-\zeta, \zeta]\cap\mathbb{Z}}(s_t)r_{\text{target}} - \mathbb I_{[-\zeta, \zeta]\cap\mathbb{Z}}(s_t),
\end{array}
\]
where $\mathbb I$ is the indicator function.
Let $R_{\text{tot}}$ be the random variable of total reward and let $J$ be the expected total reward:
$$
J[\pi(\cdot)] = \E{\sum_{t = 0}^\infty r(S_t, \pi(S_t), W_t)} = \E{R_{\text{tot}}}.
$$
The following theorem states that the objective is well-posed.
\begin{theorem}
\label{ExistenceOfExpectedValue}
The objective $J[\pi(\cdot)]$ exists and is finite for any admissible $\pi(\cdot)$.
\end{theorem}

\begin{proof}
Let $D_t$ denote the event that the agent perished exactly at step $t$ and let $D_\infty$ denote $\bigcap\limits_{t = 1}^\infty H_t$.
The sample space only comprises the latter events
\begin{equation}
\label{eqn:Complete}
\Omega =\bigcup\limits_{t = 1}^\infty D_t.
\end{equation}
Also,
\begin{equation}
\label{eqn:Separate}
\forall i, j \in \{\infty\}\cup\mathbb{N}\ : \ (i \neq j) \rightarrow D_i\cap D_j = \varnothing.
\end{equation}

Evidently, $D_t$ implies that the total reward is equal to exactly $r_{\text{target}} - t$ or, equivalently,
\begin{equation}
\label{eqn:DeathAndReward}
    \begin{array}{l}
        \forall t \in \mathbb{N} \ \forall \omega \in D_t \ :  \\
        \sum_{t = 0}^\infty r(S_t(\omega), \pi(S_t(\omega)), W_t(\omega)) = r_{\text{target}} - t.
    \end{array}
\end{equation}

Since $\forall t \in \mathbb{N} \ : \ D_\infty \subset H_t$ and $\PP(H_t) \leq C\sigma^t$, it is true that
\begin{equation}
\label{eqn:ImpossibleEvent}
    0 \leq \PP(D_\infty) \leq \inf\limits_t \PP(H_t) = \inf\limits_{t \geq 0} C\sigma^t = 0.
\end{equation}
In other words, $D_\infty$ is an impossible event.

Combining (\ref{eqn:Complete}), (\ref{eqn:Separate}), (\ref{eqn:DeathAndReward}) and (\ref{eqn:ImpossibleEvent}), we get the following representation of $R_{\text{tot}}$:
$$
R_{\text{tot}}(\omega) = \sum_{t = 1}^\infty (R_0 - t)I_{D_t}(\omega)
$$
whence
\begin{equation}
\label{eqn:ExpectedValue}
\begin{aligned}
& \E{R_{\text{tot}}} = \int\limits_\Omega R_{\text{tot}}(\omega) \PP{\mathrm d \omega} \\
& = \int\limits_\Omega \left(\sum_{t = 1}^\infty (r_{\text{target}} - t)I_{D_t}(\omega)\right) \PP{\mathrm d \omega} \\
& = \sum_{t = 1}^\infty (\zeta^* - t) \PP{D_t}.
\end{aligned}
\end{equation}

Obviously, perishing at exactly $t$ implies survival at $t - 1$, so
$$
D_t \subset S_{t - 1}
$$
whence
\begin{equation}
\label{eqn:DeathUpperBound}
    \PP{D_t} \leq \PP{H_{t - 1}}.
\end{equation}
Combining (\ref{eqn:ExpectedValue}) and (\ref{eqn:DeathUpperBound}) yields
$$
\exists \ C > 0 \ \exists \ 0 < \sigma < 1 \ : \ \E{R_{\text{tot}}} \leq \sum_{t = 1}^\infty (r_{\text{target}} - t)C\sigma^{t - 1}.
$$
Let us determine that the right-hand side of this relation converges to:
\begin{equation*}
	\begin{aligned}
		& \sum_{t = 1}^\infty (r_{\text{target}} - t)C\sigma^{t - 1} = \sum_{t = 0}^\infty (\zeta^* - 1 - t)C\sigma^t = \\
		& = C(r_{\text{target}} - 1)\sum_{t = 0}^\infty \sigma^t - C\sum_{t = 0}^\infty t\sigma^t \\
		& = \frac{C(r_{\text{target}} - 1)}{1 - \sigma} - \frac{C\sigma}{(1 - \sigma)^2}.
	\end{aligned}
\end{equation*}
Finally,
$$
\exists \ C > 0 \ \exists \ 0 < \sigma < 1 \ : \ \E{R_{\text{tot}}} \leq \frac{C(r_{\text{target}} - 1)}{1 - \sigma} - \frac{C\sigma}{(1 - \sigma)^2}
$$
which proves the claim.
\end{proof}

\section{More on algorithm convergence peculiarities in 2D heated gridworlds}\label{additional}

\vspace{-3mm}

\begin{figure}[H]
 \begin{subfigure}{\linewidth}
  \includegraphics[width=\linewidth]{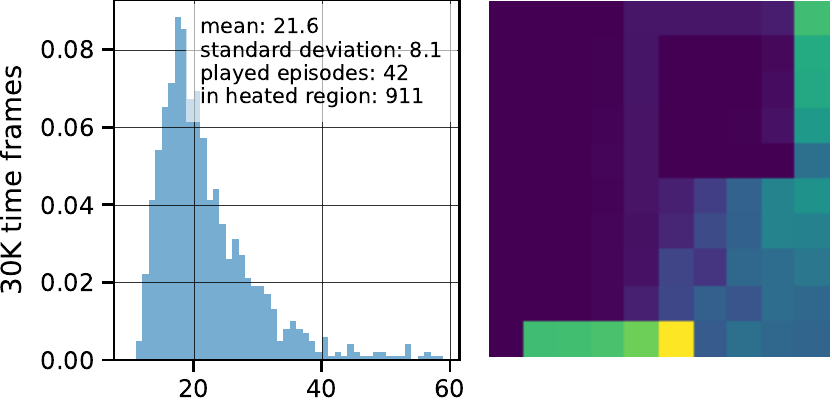}
 \end{subfigure}
 \begin{subfigure}{\linewidth}
  \includegraphics[width=\linewidth]{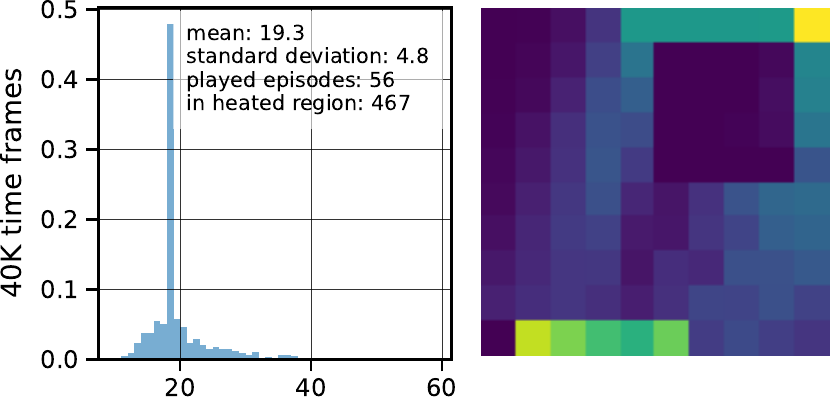}
 \end{subfigure}
 \begin{subfigure}{\linewidth}
  \includegraphics[width=\linewidth]{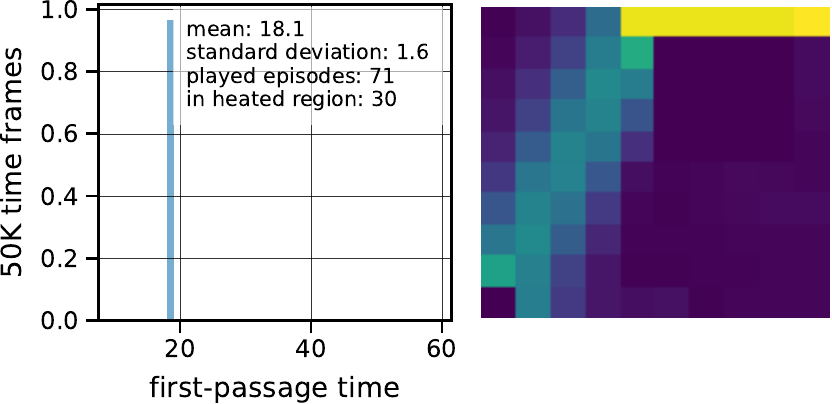}
 \end{subfigure}
 \caption{Transition from heated to deterministic route in the gridworld of Section \ref{sec:2d-sim} when $T=3$ and $\varepsilon=1.0$ through all course of learning (which in fact means offline learning). Q-learning, 1000 agents, $\alpha=0.1, \varepsilon=1.0, \gamma=0.9$, statistics is gathered after specified number of time frames for greedy behaviour.}
 \label{fig:policy_trend2}
\end{figure}

\vspace{-6mm}

\begin{figure}[H]
    \centering
    \includegraphics[width=0.95\linewidth]{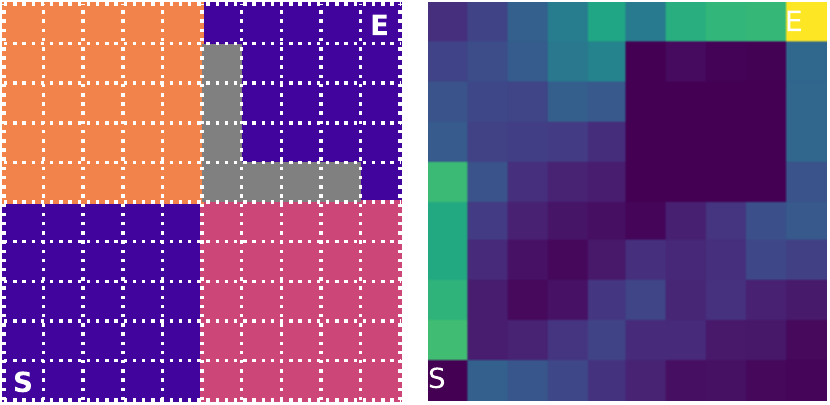}
    \caption{Path density in presence of two heated regions with temperature $T=1$ and $T=2$. 2D 10x10 gridworld, Q-learning, 500 agents, $\alpha=0.1, \varepsilon=0.1, \gamma=0.9$, learning time = 50K time frames. Agents are divided between two routes proportionally to temperature.}
    \label{fig:two_regions}
\end{figure}

\begin{figure}[H]
    \centering
    \includegraphics[width=\linewidth]{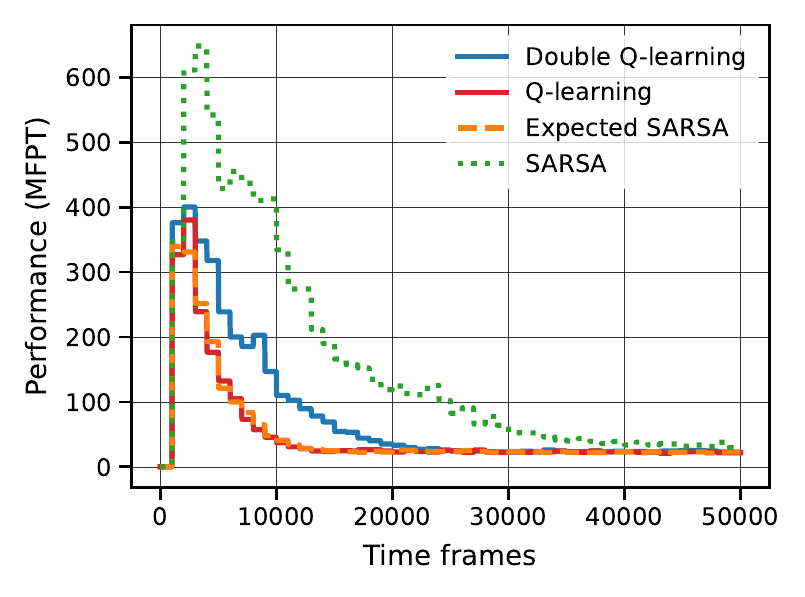}
    \caption{Performance comparison in presence of heated region with temperature $T=3$. 2D 10x10 gridworld, 100 agents, $\alpha=0.1, \varepsilon=0.1, \gamma=0.9$. Mean first-passage time after learning given number of time frames is a measure of performance (MFPT, smaller is better).}
    \label{fig:performance}
\end{figure}

\section*{Acknowledgment}

The authors acknowledge the use of computational resources of the Skoltech CDISE supercomputer Zhores for obtaining the results presented in this paper as well as the support of the Russian Science Foundation project no. 21-11-00363.

\bibliographystyle{IEEEtran}
\bibliography{references}

\begin{thebibliography}{10}
\providecommand{\url}[1]{#1}
\csname url@samestyle\endcsname
\providecommand{\newblock}{\relax}
\providecommand{\bibinfo}[2]{#2}
\providecommand{\BIBentrySTDinterwordspacing}{\spaceskip=0pt\relax}
\providecommand{\BIBentryALTinterwordstretchfactor}{4}
\providecommand{\BIBentryALTinterwordspacing}{\spaceskip=\fontdimen2\font plus
\BIBentryALTinterwordstretchfactor\fontdimen3\font minus
  \fontdimen4\font\relax}
\providecommand{\BIBforeignlanguage}[2]{{%
\expandafter\ifx\csname l@#1\endcsname\relax
\typeout{** WARNING: IEEEtran.bst: No hyphenation pattern has been}%
\typeout{** loaded for the language `#1'. Using the pattern for}%
\typeout{** the default language instead.}%
\else
\language=\csname l@#1\endcsname
\fi
#2}}
\providecommand{\BIBdecl}{\relax}
\BIBdecl

\bibitem{vision1}
H.~Rowley, S.~Baluja, and T.~Kanade, ``Neural network-based face detection,''
  \emph{IEEE Transactions on Pattern Analysis and Machine Intelligence},
  vol.~20, no.~1, pp. 23--38, 1998.

\bibitem{vision2}
P.~Viola and M.~Jones, ``Rapid object detection using a boosted cascade of
  simple features,'' in \emph{Proceedings of the 2001 IEEE Computer Society
  Conference on Computer Vision and Pattern Recognition. CVPR 2001}, vol.~1,
  2001, pp. I--I.

\bibitem{speech}
L.~Deng, G.~Hinton, and B.~Kingsbury, ``New types of deep neural network
  learning for speech recognition and related applications: an overview,'' in
  \emph{2013 IEEE International Conference on Acoustics, Speech and Signal
  Processing}, 2013, pp. 8599--8603.

\bibitem{atari}
V.~Mnih, K.~Kavukcuoglu, D.~Silver, A.~Graves, I.~Antonoglou, D.~Wierstra, and
  M.~Riedmiller, ``Playing atari with deep reinforcement learning,''
  \emph{ArXiv preprint}, p. arXiv:1312.5602, 2013.

\bibitem{Go}
D.~Silver, J.~Schrittwieser, K.~Simonyan, I.~Antonoglou, A.~Huang, A.~Guez,
  T.~Hubert, L.~Baker, M.~Lai, A.~Bolton, and et~al., ``Mastering the game of
  go without human knowledge,'' \emph{Nature}, vol. 550, no. 7676, pp.
  354--359, 2017.

\bibitem{NIPS2020VRP}
\BIBentryALTinterwordspacing
A.~Delarue, R.~Anderson, and C.~Tjandraatmadja, ``Reinforcement learning with
  combinatorial actions: An application to vehicle routing,'' in \emph{Advances
  in Neural Information Processing Systems}, H.~Larochelle, M.~Ranzato,
  R.~Hadsell, M.~F. Balcan, and H.~Lin, Eds., vol.~33.\hskip 1em plus 0.5em
  minus 0.4em\relax Curran Associates, Inc., 2020, pp. 609--620. [Online].
  Available:
  \url{https://proceedings.neurips.cc/paper/2020/file/06a9d51e04213572ef0720dd27a84792-Paper.pdf}
\BIBentrySTDinterwordspacing

\bibitem{VRPReview}
Y.~Bengio, A.~Lodi, and A.~Prouvost, ``Machine learning for combinatorial
  optimization: a methodological tour d’horizon,'' \emph{European Journal of
  Operational Research}, vol. 290, no.~2, pp. 405--421, 2020.

\bibitem{LenkaReview2019}
G.~Carleo, I.~Cirac, K.~Cranmer, L.~Daudet, M.~Schuld, N.~Tishby,
  L.~Vogt-Maranto, and L.~Zdeborov\'{a}, ``Machine learning and the physical
  sciences,'' \emph{Review of Modern Physics}, vol.~91, p. 045002, 2019.

\bibitem{carleo2019review}
\BIBentryALTinterwordspacing
G.~Carleo, I.~Cirac, K.~Cranmer, L.~Daudet, M.~Schuld, N.~Tishby,
  L.~Vogt-Maranto, and L.~Zdeborov\'a, ``Machine learning and the physical
  sciences,'' \emph{Rev. Mod. Phys.}, vol.~91, p. 045002, Dec 2019. [Online].
  Available: \url{https://link.aps.org/doi/10.1103/RevModPhys.91.045002}
\BIBentrySTDinterwordspacing

\bibitem{cichos2020review}
F.~Cichos, K.~Gustavsson, B.~Mehlig, and G.~Volpe, ``Machine learning for
  active matter,'' \emph{Nature Machine Intelligence}, vol.~2, no.~2, pp.
  94--103, 2020.

\bibitem{hydrodynamics}
S.~L. Brunton, B.~R. Noack, and P.~Koumoutsakos, ``Machine learning for fluid
  mechanics,'' \emph{Annu. Rev. Fluid Mech.}, vol.~52, pp. 477--508, 2020.

\bibitem{hasselt2010double}
H.~Hasselt, ``Double q-learning,'' \emph{Advances in neural information
  processing systems}, vol.~23, pp. 2613--2621, 2010.

\bibitem{sutton2018}
R.~S. Sutton and A.~G. Barto, \emph{Reinforcement learning: An
  introduction}.\hskip 1em plus 0.5em minus 0.4em\relax MIT press, 2018.

\bibitem{modelsactivematter}
M.~Shaebani, A.~Wysocki, R.~Winkler, and et~al., ``Computational models for
  active matter,'' \emph{Nat Rev Phys}, vol.~21, pp. 181--199, 2020.

\bibitem{biferale2019zermelo}
L.~Biferale, F.~Bonaccorso, M.~Buzzicotti, P.~Clark Di~Leoni, and
  K.~Gustavsson, ``Zermelo’s problem: Optimal point-to-point navigation in 2d
  turbulent flows using reinforcement learning,'' \emph{Chaos: An
  Interdisciplinary Journal of Nonlinear Science}, vol.~29, no.~10, p. 103138,
  2019.

\bibitem{colabrese2017flow}
S.~Colabrese, K.~Gustavsson, A.~Celani, and L.~Biferale, ``Flow navigation by
  smart microswimmers via reinforcement learning,'' \emph{Physical review
  letters}, vol. 118, no.~15, p. 158004, 2017.

\bibitem{biferale2018PRLiquids}
------, ``Smart inertial particles,'' \emph{Phys. Rev. Fluids}, vol.~3, p.
  084301, 2018.

\bibitem{collectiveswimming2018}
S.~Verma, G.~Novati, and P.~Koumoutsakos, ``Efficient collective swimming by
  harnessing vortices through deep reinforcement learning,'' \emph{Proceedings
  of the National Academy of Sciences}, vol. 115, no.~23, pp. 5849--5854, 2018.

\bibitem{muinos2021reinforcement}
S.~Mui{\~n}os-Landin, A.~Fischer, V.~Holubec, and F.~Cichos, ``Reinforcement
  learning with artificial microswimmers,'' \emph{Science Robotics}, vol.~6,
  no.~52, 2021.

\bibitem{KUKREJA2021}
\BIBentryALTinterwordspacing
N.~Kukreja and V.~K. Sharma, ``Optimal path planning of mobile nanobots
  navigation control in human physiological systems using advanced soft
  computing technique,'' \emph{Materials Today: Proceedings}, 2021. [Online].
  Available:
  \url{https://www.sciencedirect.com/science/article/pii/S2214785320406480}
\BIBentrySTDinterwordspacing

\bibitem{ocean2016}
B.~Yoo and J.~Kim, ``Path optimization for marine vehicles in ocean currents
  using reinforcement learning,'' \emph{Journal of Marine Science and
  Technology}, vol.~21, pp. 334--343, 2016.

\bibitem{reddy2016learning}
G.~Reddy, A.~Celani, T.~J. Sejnowski, and M.~Vergassola, ``Learning to soar in
  turbulent environments,'' \emph{Proceedings of the National Academy of
  Sciences}, vol. 113, no.~33, pp. E4877--E4884, 2016.

\bibitem{reddy2018glider}
G.~Reddy, J.~Wong-Ng, A.~Celani, T.~J. Sejnowski, and M.~Vergassola, ``Glider
  soaring via reinforcement learning in the field,'' \emph{Nature}, vol. 562,
  no. 7726, pp. 236--239, 2018.

\bibitem{hung2016q}
S.-M. Hung and S.~N. Givigi, ``A q-learning approach to flocking with uavs in a
  stochastic environment,'' \emph{IEEE transactions on cybernetics}, vol.~47,
  no.~1, pp. 186--197, 2016.

\bibitem{yang2020efficient}
Y.~Yang, M.~A. Bevan, and B.~Li, ``Efficient navigation of colloidal robots in
  an unknown environment via deep reinforcement learning,'' \emph{Advanced
  Intelligent Systems}, vol.~2, no.~1, p. 1900106, 2020.

\bibitem{financial}
\BIBentryALTinterwordspacing
T.~L. Meng and M.~Khushi, ``Reinforcement learning in financial markets,''
  \emph{Data}, vol.~4, no.~3, 2019. [Online]. Available:
  \url{https://www.mdpi.com/2306-5729/4/3/110}
\BIBentrySTDinterwordspacing

\bibitem{PENDHARKAR20181}
\BIBentryALTinterwordspacing
P.~C. Pendharkar and P.~Cusatis, ``Trading financial indices with reinforcement
  learning agents,'' \emph{Expert Systems with Applications}, vol. 103, pp.
  1--13, 2018. [Online]. Available:
  \url{https://www.sciencedirect.com/science/article/pii/S0957417418301209}
\BIBentrySTDinterwordspacing

\bibitem{hu2018accurate}
Z.~Hu, Y.~Jiang, X.~Ling, and Q.~Liu, ``Accurate q-learning,'' in
  \emph{International Conference on Neural Information Processing}.\hskip 1em
  plus 0.5em minus 0.4em\relax Springer, 2018, pp. 560--570.

\bibitem{he2019interleaved}
M.~He and H.~Guo, ``Interleaved q-learning with partially coupled training
  process,'' in \emph{Proceedings of the 18th International Conference on
  Autonomous Agents and MultiAgent Systems}, 2019, pp. 449--457.

\bibitem{}
R.~Fox, A.~Pakman, and N.~Tishby, ``Taming the noise in reinforcement learning
  via soft updates,'' \emph{Proceedengs of UIAI}, 2016.

\bibitem{perturbed2020}
\BIBentryALTinterwordspacing
J.~Wang, Y.~Liu, and B.~Li, ``Reinforcement learning with perturbed rewards,''
  in \emph{Proceedings of the AAAI Conference on Artificial Intelligence},
  vol.~04.\hskip 1em plus 0.5em minus 0.4em\relax AAAI Press, Palo Alto,
  California USA, 2020, pp. 6202--6209. [Online]. Available:
  \url{https://doi.org/10.1609/aaai.v34i04.6086}
\BIBentrySTDinterwordspacing

\bibitem{everitt2017corruptedreward}
\BIBentryALTinterwordspacing
T.~Everitt, V.~Krakovna, L.~Orseau, and S.~Legg, ``Reinforcement learning with
  a corrupted reward channel,'' in \emph{IJCAI}, 2017, pp. 4705--4713.
  [Online]. Available: \url{https://doi.org/10.24963/ijcai.2017/656}
\BIBentrySTDinterwordspacing

\bibitem{loftin2014humanfeedback}
R.~Loftin, B.~Peng, J.~MacGlashan, M.~L. Littman, M.~E. Taylor, J.~Huang, and
  D.~L. Roberts, ``Learning something from nothing: Leveraging implicit human
  feedback strategies,'' in \emph{The 23rd IEEE international symposium on
  robot and human interactive communication}.\hskip 1em plus 0.5em minus
  0.4em\relax IEEE, 2014, pp. 607--612.

\bibitem{schoettler2020insertiontask}
G.~Schoettler, A.~Nair, J.~Luo, S.~Bahl, J.~Aparicio~Ojea, E.~Solowjow, and
  S.~Levine, ``Deep reinforcement learning for industrial insertion tasks with
  visual inputs and natural rewards,'' in \emph{2020 IEEE/RSJ International
  Conference on Intelligent Robots and Systems (IROS)}, 2020, pp. 5548--5555.

\bibitem{gullapalli1994peginhole}
V.~Gullapalli, J.~A. Franklin, and H.~Benbrahim, ``Acquiring robot skills via
  reinforcement learning,'' \emph{IEEE Control Systems Magazine}, vol.~14,
  no.~1, pp. 13--24, 1994.

\bibitem{johannink2019residual}
T.~Johannink, S.~Bahl, A.~Nair, J.~Luo, A.~Kumar, M.~Loskyll, J.~A. Ojea,
  E.~Solowjow, and S.~Levine, ``Residual reinforcement learning for robot
  control,'' in \emph{2019 International Conference on Robotics and Automation
  (ICRA)}.\hskip 1em plus 0.5em minus 0.4em\relax IEEE, 2019, pp. 6023--6029.

\bibitem{howell1997vehiclesuspension}
M.~N. Howell, G.~P. Frost, T.~J. Gordon, and Q.~H. Wu, ``Continuous action
  reinforcement learning applied to vehicle suspension control,''
  \emph{Mechatronics}, vol.~7, no.~3, pp. 263--276, 1997.

\bibitem{Huang2017AdversarialAO}
S.~H. Huang, N.~Papernot, I.~Goodfellow, Y.~Duan, and P.~Abbeel, ``Adversarial
  attacks on neural network policies,'' \emph{ArXiv}, vol. abs/1702.02284,
  2017.

\bibitem{Huang2019DeceptiveRL}
Y.~Huang and Q.~Zhu, ``Deceptive reinforcement learning under adversarial
  manipulations on cost signals,'' in \emph{GameSec}, 2019.

\bibitem{Gleave2020AdversarialPA}
A.~Gleave, M.~Dennis, N.~Kant, C.~Wild, S.~Levine, and S.~J. Russell,
  ``Adversarial policies: Attacking deep reinforcement learning,''
  \emph{ArXiv}, vol. abs/1905.10615, 2020.

\bibitem{baird1994reinforcement}
L.~C. Baird, ``Reinforcement learning in continuous time: Advantage updating,''
  in \emph{Proceedings of 1994 IEEE International Conference on Neural Networks
  (ICNN'94)}, vol.~4.\hskip 1em plus 0.5em minus 0.4em\relax IEEE, 1994, pp.
  2448--2453.

\bibitem{fox2015taming}
\BIBentryALTinterwordspacing
R.~Fox, A.~Pakman, and N.~Tishby, ``Taming the noise in reinforcement learning
  via soft updates,'' in \emph{Proceedings of conference on uncertainty in
  artificial intelligence (UAI)}, A.~Ihler and D.~Janzing, Eds., vol.~32.\hskip
  1em plus 0.5em minus 0.4em\relax AUAI Press, 2016, pp. 202--211. [Online].
  Available: \url{https://auai.org/uai2016/proceedings/papers/219.pdf}
\BIBentrySTDinterwordspacing

\bibitem{Qlearning}
C.~Watkins, ``Learning from delayed rewards,'' Ph.D. dissertation, King’s
  College, Cambridge, 1989.

\bibitem{redner01}
S.~Redner, \emph{A guide to first-passage processes}.\hskip 1em plus 0.5em
  minus 0.4em\relax Cambridge: Cambridge University Press, 2001.

\bibitem{Grebenkov_2020}
\BIBentryALTinterwordspacing
D.~S. Grebenkov, D.~Holcman, and R.~Metzler, ``Preface: new trends in
  first-passage methods and applications in the life sciences and
  engineering,'' \emph{Journal of Physics A: Mathematical and Theoretical},
  vol.~53, no.~19, p. 190301, apr 2020. [Online]. Available:
  \url{https://doi.org/10.1088/1751-8121/ab81d5}
\BIBentrySTDinterwordspacing

\bibitem{viswanathan}
G.~M. Viswanathan, M.~G.~E. da~Luz, E.~P. Raposo, and H.~E. Stanley, \emph{The
  Physics of Foraging: An Introduction to Random Searches and Biological
  Encounters}.\hskip 1em plus 0.5em minus 0.4em\relax Cambridge University
  Press, 2011.

\bibitem{Mirny_2009}
\BIBentryALTinterwordspacing
L.~Mirny, M.~Slutsky, Z.~Wunderlich, A.~Tafvizi, J.~Leith, and A.~Kosmrlj,
  ``How a protein searches for its site on {DNA}: the mechanism of facilitated
  diffusion,'' \emph{Journal of Physics A: Mathematical and Theoretical},
  vol.~42, no.~43, p. 434013, oct 2009. [Online]. Available:
  \url{https://doi.org/10.1088/1751-8113/42/43/434013}
\BIBentrySTDinterwordspacing

\bibitem{risk}
\BIBentryALTinterwordspacing
J.~Masoliver and J.~Perell\'o, ``First-passage and risk evaluation under
  stochastic volatility,'' \emph{Phys. Rev. E}, vol.~80, p. 016108, Jul 2009.
  [Online]. Available:
  \url{https://link.aps.org/doi/10.1103/PhysRevE.80.016108}
\BIBentrySTDinterwordspacing

\bibitem{stoshasticgames}
\BIBentryALTinterwordspacing
C.-Y. Wei, Y.-T. Hong, and C.-J. Lu, ``Online reinforcement learning in
  stochastic games,'' in \emph{Advances in Neural Information Processing
  Systems}, I.~Guyon, U.~V. Luxburg, S.~Bengio, H.~Wallach, R.~Fergus,
  S.~Vishwanathan, and R.~Garnett, Eds., vol.~30.\hskip 1em plus 0.5em minus
  0.4em\relax Curran Associates, Inc., 2017. [Online]. Available:
  \url{https://proceedings.neurips.cc/paper/2017/file/36e729ec173b94133d8fa552e4029f8b-Paper.pdf}
\BIBentrySTDinterwordspacing

\bibitem{reachability}
S.~Debnath, L.~Liu, and G.~Sukhatme, ``Solving markov decision processes with
  reachability characterization from mean first passage times,'' in \emph{2018
  IEEE/RSJ International Conference on Intelligent Robots and Systems (IROS)},
  2018, pp. 7063--7070.

\bibitem{Palyulin_2012}
\BIBentryALTinterwordspacing
V.~V. Palyulin and R.~Metzler, ``How a finite potential barrier decreases the
  mean first-passage time,'' \emph{Journal of Statistical Mechanics: Theory and
  Experiment}, vol. 2012, no.~03, p. L03001, mar 2012. [Online]. Available:
  \url{https://doi.org/10.1088/1742-5468/2012/03/l03001}
\BIBentrySTDinterwordspacing

\bibitem{Chupeau1383}
\BIBentryALTinterwordspacing
M.~Chupeau, J.~Gladrow, A.~Chepelianskii, U.~F. Keyser, and E.~Trizac,
  ``Optimizing brownian escape rates by potential shaping,'' \emph{Proceedings
  of the National Academy of Sciences}, vol. 117, no.~3, pp. 1383--1388, 2020.
  [Online]. Available: \url{https://www.pnas.org/content/117/3/1383}
\BIBentrySTDinterwordspacing

\bibitem{bertsekas1996neuro}
D.~P. Bertsekas and J.~N. Tsitsiklis, \emph{Neuro-dynamic programming}.\hskip
  1em plus 0.5em minus 0.4em\relax Athena Scientific, 1996.

\bibitem{bellman1957dynamic}
\BIBentryALTinterwordspacing
R.~BELLMAN, ``A markovian decision process,'' \emph{Journal of Mathematics and
  Mechanics}, vol.~6, no.~5, pp. 679--684, 1957. [Online]. Available:
  \url{http://www.jstor.org/stable/24900506}
\BIBentrySTDinterwordspacing

\bibitem{watkins1992q}
C.~J. Watkins and P.~Dayan, ``Q-learning,'' \emph{Machine learning}, vol.~8,
  no. 3-4, pp. 279--292, 1992.

\bibitem{rummery1994sarsa}
G.~A. Rummery and M.~Niranjan, \emph{On-line Q-learning using connectionist
  systems}.\hskip 1em plus 0.5em minus 0.4em\relax University of Cambridge,
  Department of Engineering Cambridge, UK, 1994, vol.~37.

\bibitem{john1994expected_sarsa}
G.~H. John, ``When the best move isn't optimal: Q-learning with exploration,''
  in \emph{AAAI}, vol. 1464.\hskip 1em plus 0.5em minus 0.4em\relax Citeseer,
  1994.

\end{thebibliography}

\end{document}